\documentclass{article}

\usepackage{arxiv}
\usepackage{geometry}
\usepackage{graphicx} 
\usepackage{amsmath} 
\usepackage{amsthm}   
\usepackage{amsfonts}   
\usepackage{amssymb}  
\usepackage{latexsym}  
\usepackage{float} 
\usepackage{verbatim}
\usepackage{appendix}
\usepackage{float}     
\usepackage{epstopdf}   
\usepackage{color}      
\usepackage{mathrsfs}
\usepackage{hyperref}
\usepackage{indentfirst}
\usepackage{enumerate}

\allowdisplaybreaks[4]
\numberwithin{equation}{section}

\renewcommand\appendix{\par
    \setcounter{section}{0}
    \gdef\thesection{Appendix~ \Alph{section}}
    \setcounter{equation}{0}
     \renewcommand\theequation{A.~\arabic{equation}}} 

\newcommand{\PP}{\mathbb{P}}
\newcommand{\RR}{\mathbb{R}}

\theoremstyle{definition}
\newtheorem{theorem}{Theorem}[section]
\newtheorem{lemma}[theorem]{Lemma}

\newtheorem{definition}[theorem]{Definition}
\newtheorem{example}{Example}

\newtheorem{assumption}[theorem]{Assumption}

\title{Implicit numerical approximation for stochastic delay differential equations with the nonlinear diffusion term in the infinite horizon}

\author{
 Yudong Wang and Hongjiong Tian \\
 Department of Mathematics, Shanghai Normal University, Shanghai, 200234, China\\
\texttt{  wangyudong\_edu@163.com; hjtian@shnu.edu.cn}
}

\begin{document}
\maketitle

\begin{abstract}
    This paper investigates the approximation of stochastic delay differential equations (SDDEs) via the backward Euler-Maruyama (BEM) method under generalized monotonicity and Khasminskii-type conditions in the infinite horizon. First, by establishing the uniform moment boundedness and finite-time strong convergence of the BEM method, we prove that for sufficiently small step sizes, the numerical approximations strongly converge to the underlying solution in the infinite horizon with a rate of $1/2$, which coincides with the optimal finite-time strong convergence rate.
    Next, we establish the uniform boundedness and convergence in probability for the segment processes associated with the BEM method. This analysis further demonstrates that the probability measures of the numerical segment processes converge to the underlying invariant measure of the SDDEs. Finally, a numerical example and simulations are provided to illustrate the theoretical results.
\end{abstract}

\section{Introduction}
Due to the widespread occurrence of time delays and environmental noises in real-world systems, stochastic delay differential equations (SDDEs) have emerged as a crucial mathematical model in the realms of science and engineering in recent decades \cite{AHMG2007,  CGMP2019, LM2006, M2007, M1974}. However, obtaining explicit solutions for SDDEs is challenging, leading to an increasing focus on numerical approximations in recent years.

For numerical methods, a critical focus is on studying their strong convergence. Many ignificant results have been obtained, including but not limited to: the classic Euler-Maruyama method \cite{Buckwar2000, MS2003}, the backward Euler-Maruyama method \cite{LCF2004,Zhou2015}, the truncated Euler-Maruyama method \cite{GMY2018, SHGL2022}, the tamed Euler method \cite{JY2017}, etc. It should be noted that the above works primarily focus on the strong convergence analysis in finite time. Specifically, for the solution of SDDE, $\{x(t)\}_{t\geq -\tau}$, and its corresponding continuous numerical solution, $\{X(t)\}_{t \geq -\tau}$, the classical study of strong convergence aims to establish upper bounds of the form
\begin{align*}
    \sup \limits_{0 \leq t \leq T}\mathbb{E}|x(t)-X(t)|^p \leq C_T \Delta^q,
\end{align*}
where $T>0, p>0, q>0$ are constants, $\Delta$ is the step size and $C_T$ is a constant dependent on $T$. In general, due to the application of tools such as the Gronwall lemma, the typical form of $C_T$ is an exponential function related to time, such as $C_T=ce^{cT}$ for some positive constant $c$ \cite{LCF2004, SHGL2022}, which implies that $C_T$ is a function that grows over time.

However, this poses some problems. For instance, when dealing with certain issues that are difficult to determine within a fixed time interval, such as the timing of investment opportunities based on favorable market conditions \cite{DGM2021,GA1997}, challenges arise because these opportunities may occur at any moment. Therefore, any assumption that presupposes a specific time interval, including conducting convergence analysis within a finite time, may not be appropriate.

Thus, unlike most existing results, one of the main aims of this paper is devoted to  the study of the strong convergence of backward Euler-Maruyama (BEM) method in infinite horizon, that is, to try to obtain a constant $C$ independent of $T$, such that
\begin{align}\label{UIT}
    \sup \limits_{t \geq 0}\mathbb{E}|x(t)-X(t)|^p \leq C \Delta^q.
\end{align}
If \eqref{UIT} holds, One of the simplest and most straightforward applications is that we no longer need to adjust the step size to meet the specified threshold accuracy when numerically simulating the behavior of the system over a long period of time. It is worth mentioning that progress has been made in this direction by Crisan et al. \cite{ACO2023, CDO2021}, who have developed a method for analyzing the convergence of various numerical methods in the infinite horizon. However, unlike theirs, our focus is on the strong convergence of stochastic differential delay equations in infinite horizons, rather than on weak convergence.

In addition, due to the presence of time delay, the solutions of SDDEs are not Markov processes, and therefore, properties based on Markov processes do not apply. At this point, our focus should shift to the segment process associated with the solution — a continuous function-valued Markov process \cite{M1986}. Therefore, it is sometimes necessary to analyze the convergence of the segment process associated with the numerical method \cite{BSY-2023, LMS2023}. Inspired by these works, the other main aim of this paper is to study the uniform convergence of numerical segment processes in probability and further prove that the probability measure corresponding to the numerical method converges to the underlying invariant measure related to SDDEs in the Bounded Lipschitz metric. Meanwhile, as both an independent result and a prerequisite for the convergence analysis, we also establish the moment boundedness of the numerical solution.

 The main contributions of this paper are summarized as follows:
\begin{itemize}
    \item Inspired by works of Mao (see e.g. \cite{HMS2003,Mao2007, Mao2015}), we analyze the strong convergence of BEM in infinite horizon, extending beyond traditional finite-time frameworks.
    \item We establish the uniform boundedness and convergence in probability of the numerical segment process.
    \item For SDDEs with nonlinear drift and diffusion coefficients, we prove that the probability measure of the numerical segment process converges to the underlying invariant measure.
\end{itemize}

The rest of this paper is organized as follows. Section 2 introduces the necessary notations. Section 3 discusses the necessary properties of the underlying solution. The main results are presented in Section 4, and Section 5 provides a numerical simulation to illustrate the theoretical findings.

\section{Mathematical Preliminaries}
At the beginning of this section, we introduce some necessary notation.
Let $|\cdot|$ denote the Euclidean norm in $\mathbb{R}^d$ and the trace norm in $\mathbb{R}^{d \times m}$, and let $\langle \cdot, \cdot \rangle$ be the inner product in $\mathbb{R}^d$.
For real numbers $a$ and $b$, define $a \vee b = \max\{a, b\}$ and $a \wedge b = \min\{a, b\}$.
Let $\lfloor a \rfloor$ denote the integer part of $a$, and let $\boldsymbol{1}_A(x)$ be the indicator function of the set $A$.
Define $\mathbb{R}^+ = [0, \infty)$.
Let $\mathcal{C} := C([-\tau, 0]; \mathbb{R}^d)$ denote the family of continuous functions $X: [-\tau, 0] \to \mathbb{R}^d$ equipped with the supremum norm $\|X\| = \sup_{-\tau \leq \theta \leq 0} |X(\theta)|$.
For $R>0$, define $B(R):=\{X\in \mathcal{C}:\|X\|\leq R\}$.
For $p \geq 2$, define $\mathcal{C}_{\mathcal{F}_0}^p := \mathcal{C}_{\mathcal{F}_0}^p([-\tau, 0]; \mathbb{R}^d)$ as the family of all $\mathcal{F}_0$-measurable bounded $C$-valued random variables $\xi = \{\xi(\theta): -\tau \leq \theta \leq 0\}$.
Finally, let $\hat{\mathcal{C}}(\mathbb{R}^d \times \mathbb{R}^d; \mathbb{R}^+)$ denote the family of continuous functions $V: \mathbb{R}^d \times \mathbb{R}^d \to \mathbb{R}^+$ satisfying $V(x, x) = 0$ for all $x \in \mathbb{R}^d$.

This paper focuses on the scalar SDDE of the form
\begin{align}\label{SDDE}
dx(t)=f(x(t),x(t-\tau))dt +g(x(t),x(t-\tau))dW(t),~~~t> 0,
\end{align}
with the initial data $\xi=\{\xi(\theta)$: $\theta\in [-\tau, 0]\} \in \mathcal{C}$,
where  $\tau>0$,  $f:  \RR^d \times \RR^d \rightarrow \RR^{d}$ and $g: \RR^d\times \RR^d  \rightarrow \RR^{d\times m}$ are Borel measurable, $W(t)$ is an m-dimensional Brownian motion on a probability space  $( \Omega,~\mathcal{F}, ~\PP )$ with a right-continuous complete filtration $\{{\mathcal{F}}_{t}\}_{t\geq 0}$, and $\xi$ is an $\mathcal F_0$-measurable, continuous function-valued random variable from
$[-\tau,~0]$ to $\mathbb{R}^{d}$.
Let $\{x_{t}\}_{t\geq 0} $ be the  segment process, where $x_{t}(\theta):=x(t+\theta)$ for $\theta\in[-\tau,0]$, and sometimes to emphasize the initial value $\xi$ at $t=0$, we also write the solution of \eqref{SDDE} as $x(t; 0, \xi)$  while the corresponding $x_t(\theta)$ is written as $x_t^{0, \xi}(\theta)$. It must be mentioned that the process $\{x_t\}_{t \geq 0}$ associated with the underlying solution of $\eqref{SDDE}$ is a homogeneous Markov process; see, for example, Theorem 1.1 in \cite{M1986} or Proposition 3.4 in \cite{RRV2006} for more details.

For simplicity, given $T, T_1, T_2, \tau \in \mathbb{R}_{+}$, we may assume without loss of generality that $\Delta=\frac{T}{N}=\frac{T_1}{N_1}=\frac{T_2}{N_2}=\frac{\tau}{M}$ for some integers $N$, $N_1$, $N_2$ and $M$, and let $t_k = k\Delta$ for $k \geq -M$. Throughout the paper, $C$ denotes a generic positive constant that may vary in different contexts. Then, we make the following assumptions for the analysis.

\begin{assumption}\label{as1}
 There exist constants $a_1 $, $a_2$, $a_3, a_4, a_5, a_6, a_7, a_8, a_9$, $q$ with $q\geq 2$ and a function $V(\cdot, \cdot)\in \hat{\mathcal{C}}(\mathbb{R}^d \times \mathbb{R}^d; \mathbb{R}^+)$ such that
\begin{align*}
  | f(x, y)-f(\bar{x},\bar  y)| \leq a_1(1+|x|^{q-1}&+|\bar{x}|^{q-1}+|y|^{q-1}+|\bar{y}|^{q-1}) (| x-\bar {x}|+| y-\bar {y}|) \\
  |g(x,y)-g(\overline{x},\overline{y})|^2 &\leq \epsilon_1|g_1(x)-g_1(\overline{x})|^2+\epsilon_2|g_2(y)-g_2(\overline{y})|^2\\
  \epsilon_1|g_1(x)-g_1(\overline{x})|^2&\leq a_2 |x-\overline{x}|^2+a_4V(x,\overline{x})\\
  \epsilon_2|g_2(y)-g_2(\overline{y})|^2&\leq a_3|y-\overline{y}|^2 +a_5V(y,\overline{y})\\
  |f(x, y)|&\leq a_6 (1+|x|^q+|y^q|) , \\
  |g(x,y)|^2&\leq a_7+a_8|x|^{q+1}+a_9|y|^{q+1})
 \end{align*}
 for any  $x,~\bar{x},~y,~\bar y\in \RR ^d $ , where $V(x,y)=|x-y|^2(|x|^{q-1}+|y|^{q-1})$
\end{assumption}
\begin{assumption}\label{as2}
   There exist nonnegative constants  $b_1-b_{13},~l_{1}, ~l_2$ and $\nu$ with $b_1>b_2+l_1(a_2+a_3),~b_3>l_1(a_4+a_5),~b_5>b_6,~b_7>b_8, ~b_{10}>b_{11},~b_{12}>b_{13}, ~l_1> 2, ~l_2>1$ and $p^* =4q-2$,  such that
\begin{align*}
2\big\langle x-\bar {x} ,f(x,&y)-f(\bar {x},\bar y)\big\rangle
\leq -b_1| x-\bar {x}|^2+b_2| y-\bar {y}| ^2 - b_3V(x,\overline{x})\\
\frac{p^*}{2}\big|x\big|^{p^*-2}\big(2\big \langle x, f(x,y) \big \rangle & + (p^*-1)|g(x,y)|^2 \big)\leq b_4-b_5|x|^{p^*}+b_6|y|^{p^*} - b_7|x|^{p^*+q-1}+b_8|y|^{p^*+q-1} \\
2\big \langle x, f(x,y)& \big \rangle  + l_2|g(x,y)|^2 \leq b_{9}-b_{10}|x|^{2}+b_{11}|y|^2-b_{12}|x|^{q+1}+b_{13}|y|^{q+1}
\end{align*}
for any $x,~\bar x,~y,~\bar y\in\RR^d$
\end{assumption}

\begin{assumption} \label{as3}
     The initial value is Holder continuous, i.e., there exist a positive constant $K_1$ such that for any $-\tau \leq s < t \leq 0$
     \begin{align}
         |\xi(t)-\xi(s)|^2 \leq K_1(t-s).
     \end{align}
\end{assumption}

It should be pointed out that under Assumptions \ref{as1} and \ref{as2}, SDDE \eqref{SDDE} with the initial data $\xi\in\mathcal {C}_{\mathcal{ F}_0}^{p^*+q-1}$ has a unique global solution $x(t)$ for $t\geq -\tau$ with $\mathbb{E}(|x(t)|^{p^*})\leq \infty$ (see \cite[Theorem 2.4]{MR2005} or \cite[p.278, Theorem 7.12]{MY2006} for details).

Now, we define the backward Euler-Maruyama approximate solution for Eq. \ref{SDDE} as follows.
\begin{align}\label{DSDDE}
    \begin{split}
  \left \{
 \begin{array}{ll}
X_k=\xi(t_k), \quad\quad\quad\quad k=-M,-M+1,...,0,\\
 X_k=X_{k-1}+f(X_k, X_{k-M})\Delta+ g(X_{k-1}, X_{k-M-1})\Delta W_{k-1}, \quad \quad k=1,2,...,
 \end{array}
 \right.
 \end{split}
 \end{align}
 where $\Delta W_{k-1} := W(k\Delta)-W((k-1)\Delta)$. Clearly, the BEM method \eqref{DSDDE} is well defined (see e.g. \cite{SWMF2025}).
Then, the continuous backward Euler-Maruyama approximate solution is defined by
\begin{align}\label{CDSDDE}
    \begin{split}
  \left \{
 \begin{array}{ll}
X(t)=\xi(t), \quad\quad\quad\quad\quad\quad\quad\quad\quad\quad\quad 0 \geq t\geq-\tau ,\\
X(t)=X_k+\theta(X_{k+1}-X_k), \quad\quad\quad\quad\quad~ t\geq 0
 \end{array}
 \right.
 \end{split}
 \end{align}
 where $k=\lfloor t/\Delta\rfloor$ and $\theta=t/\Delta-k$. Similar to the underlying solution of Eq.\eqref{SDDE}, denote $X_{t}:=\{X(t+\theta) : -\tau \leq \theta \leq 0\}$, and sometimes to emphasize the initial value $\xi$ at $t=0$, we also denote them by $X_{k}^{0, \xi}$ and $X_{t}^{0, \xi}$.

Next, let $\mathcal{F}_t:= \sigma (W(u), 0 \leq u \leq t) \wedge \mathcal{N}$ for $t > 0$, where $\mathcal{N}$ is the set of all $\mathbb{P}$-null sets.
In a way similar to the proof of \cite[Theorem 6.1]{SWMF2025} , we obtain the following result.
 \begin{lemma}\label{L--5}
 Under Assumption \ref{as2}, $\{X_{t_k}\}_{k \geq 0}$ is a homogeneous Markov chain, i.e., for any $B \in \mathfrak{B}(C)$ and $\xi \in \mathcal{C}$
  \begin{align}\label{TH}
      \mathbb{P}\big(~X_{t_{k+1}} \in B |~X_{t_k} = \xi  \big)= \mathbb{P}\big(~X_{t_1} \in B |~X_{0}= \xi \big),
  \end{align}
  and
  \begin{align}\label{MP}
       \mathbb{P}\big(~X_{t_{k+1}} \in B |~\mathcal{F}_{t_k} \big)= \mathbb{P}\big(~X_{t_{k+1}} \in B |~X_{t_k} \big).
  \end{align}
 \end{lemma}

Moreover, to explore the invariant measure of the solution, some necessary concepts also need to be introduced. Let $\mathfrak{B}(C)$ be the Borel algebra of $C$ and $\mathcal{P}(C):=\mathcal{P}(C, \mathfrak{B}(C))$ be the family of  probability measures on $(C,\mathfrak{B}(C))$. For any $\xi \in C$ and $B \in \mathfrak{B}(C)$, the transition probabilities for the segment processes $x_t$ and $X_{t_k}$ with initial value $x_0=X_0 =\xi$ is defined as
\begin{align*}
    \PP_t(\xi, B) :=\PP(x_t \in B | x_0 \in \xi)  \quad\quad \overline{\PP}_{t_k}(\xi, B) :=\PP(X_{t_k} \in B | X_0 \in \xi)
\end{align*}
Define metric $d_{\mathbb{L}}$ on $\mathcal{P}(C)$ by
\begin{equation*}
    \mathit{d}_{\mathbb{L}}(\mathbb{P}_1 , \mathbb{P}_2) = \sup \limits_{\mathit{F} \in \mathbb{L}} \left| \int_{C} \mathit{F}(x) \mathbb{P}_1(dx) - \int_{C} \mathit{F}(x)\mathbb{P}_2(dx) \right|,
\end{equation*}
where $\mathbb{L}$ is the test functional space
\begin{align*}
    \mathbb{L} := \Big\{F: C \to \mathbb{R} \Big| |F(x)-F(y)| \leq |x-y| \quad\text{and}\quad \sup \limits_{x\in C} |F(x)| \leq 1 \Big\}
\end{align*}
The weak convergence of probability measures can be illustrated in terms of metric $\mathit{d}_{\mathbb{L}}$ \cite{IW1989}. That is, a sequence of probability measures $\{\mathbb{P}_t\}_{t \geq 0}$ in $\mathcal{P}(C)$ converge weakly to a probability measure $\mathbb{P} \in \mathcal{P}(C)$ if and only if
\begin{equation*}
    \lim \limits_{ k \to \infty} \mathit{d}_{\mathbb{L}}(\mathbb{P}_k, \mathbb{P}) = 0.
\end{equation*}
Then we define the invariant measure by using the concept of weak convergence.

\begin{definition}
 For any initial value $\xi \in \mathcal{C}$, the $C$-valued process $\{x_t^{0, \xi}\}_{t \geq 0}$ is said to have a invariant measure $\pi \in \mathcal{P}(C)$ if the transition probability measure ${\mathbb{P}}_t(\xi,\cdot)$ converges weakly to $\pi( \cdot)$ as $t \to \infty$ for every $\xi \in C$, that is
    \begin{equation*}
        \lim \limits_{t \to \infty} \left(\sup \limits_{\mathit{F} \in \mathbb{L}} \left|\mathbb{E}(\mathit{F}(x_t^{0, \xi}))-\mathbb{E}_{\pi}(\mathit{F})\right|\right) = 0,
    \end{equation*}
    where
    \begin{equation*}
        \mathbb{E}_{\pi}(\mathit{F}) = \int_{C} \mathit{F}(y)\pi(dy).
    \end{equation*}
\end{definition}
We end up this section by pointing out some crucial inequalities for the analysis in our paper.
For any $a_i \in \mathbb{R}^d, n_i >0, k\in \mathbb{N},$ and $ p\geq 1, $ the inequalities
\begin{align}\label{EI-1}
   | \sum \limits_{i=1}^k a_i |^p \leq k^{p-1}\sum \limits_{i=1}^k |a_i|^p
 \end{align}
 and
 \begin{align}\label{EI-2}
     |a_1|^{n_2} \leq |a_1|^{n_1} +|a_1|^{n_3}
 \end{align}
 hold, where $n_1 \leq n_2 \leq n_3$

\section{Some properties of the underlying solution}
 In this section, we mainly introduce some properties of the solution to \eqref{SDDE} for later analysis.
   Following similar arguments as in Lemma 3.1, Lemma 3.2 and Theorem 3.6 in \cite{WWM2019}, we obtain the following results. The first property we will show is boundedness in probability of the underlying solution \eqref{SDDE}.
\begin{lemma}\label{LM3.1}
    Under Assumptions \ref{as1} and \ref{as2}, for any $\xi \in B(R)$, $ t\geq 0$, $T >0$ and $ \epsilon_3 >0$, there exists a positive constant $\hat{H}$ such that
    \begin{align*}
        \mathbb{P}\{|x(s;0,\xi)| \leq \hat{H}, \forall s \in [t, t+T]\} \geq 1-\epsilon_3,
    \end{align*}
    for all $t \geq -\tau$.
\end{lemma}
The next result will be used in the analysis of Lemma \ref{LM47}.
\begin{lemma}\label{LM3.2}
    For any  $\xi \in B(M)$ and $\epsilon_4, \epsilon_5 > 0$, there exists a positive constant $\delta$ such that
    \begin{align*}
        \mathbb{P}\{\sup \limits_{\substack{t_2-t_1 \leq \delta\\
t-\tau \leq t_1 \leq t_2 \leq t} } |x(t_2; 0, \xi)-x(t_1;0,\xi)| \geq \epsilon_4  \} \leq \epsilon_5,
    \end{align*}
    for any $t \geq 0$.
\end{lemma}
The final result will be used in Theorem \ref{TM48}.
\begin{lemma}\label{LM3.3}
    Under Assumption \ref{as1}-\ref{as3}. The segment process $\{x_t\}_{t \geq 0}$ is asymptotically stable in distribution and admits a unique invariant measure $\pi$.
\end{lemma}

  \section{Main results}
  In this section, we will establish a series of properties of the numerical solution produced by the BEM method. We begin by showing that the moments of the numerical solution are bounded.
  \begin{lemma}\label{LM57}
    Under Assumption \ref{as2}, for any $\xi\in B(R)$, there exists a $\Delta_1>0$ such that for any $\Delta \in(0,\Delta_1]$, the BEM method \eqref{DSDDE} has the following property
      \begin{align*}
       \sup \limits_{i\geq -M} \mathbb{E}\big(|X_i|^2+\frac{l_2\Delta}{1+\lambda_1 \Delta}|g(X_i,X_{i-M})|^2\big)+ \sup\limits_{k\geq0}\Delta \sum\limits_{j=k-M+1}^{k} \mathbb{E}(|X_j|^{q+1})\leq C,
   \end{align*}
   where $\lambda_1$ is a constant such that $b_{10}-\lambda_1-b_{11}e^{\lambda_1\tau}>0$ and $b_{12}-b_{13}e^{\lambda_1\tau}>0$, $C$ is a constant independent of $k$ but dependent on the initial value $\xi$ and $\lambda$.
\end{lemma}
\begin{proof}
    First, define a stopping time $\lambda_{\Lambda}=\inf\{k: X_k > \Lambda\}$. Then, from \eqref{DSDDE}, we can see that for any $k \in \mathbb{N}$
    \begin{align}\label{eq---4.3}
      \notag  |X_{k+1}|^2-|X_k|^2 + |X_{k+1}-X_k|^2&=2\big\langle X_{k+1}, X_{k+1}-X_k \big\rangle \\ \notag
        &\leq 2\langle X_{k+1}, f(X_{k+1},X_{k+1-M})\Delta \rangle + |X_{k+1}-X_k|^2+ |g(X_{k},X_{k+1})\Delta W_{k}|^2\\
        &\quad~+ 2\langle X_k, g(X_{k},X_{k-M})\Delta W_k \rangle.
    \end{align}
  Since $b_{10}>b_{11}$ and $b_{12}>b_{13}$, we choose a sufficiently small $\lambda>0$ such that $b_{10}-\lambda_1-b_{11}e^{\lambda_1\tau}>0$ and $b_{12}-b_{13}e^{\lambda_1\tau}>0$. It then follows from \eqref{eq---4.3}, Assumptions \ref{as1} and \ref{as2} that
    \begin{align}\label{eq--4.4}
     \notag   \mathbb{E}&\Big((1+\lambda_1\Delta)|X_{(k+1)\wedge \lambda_{\Lambda}}|^2+l_2\Delta |g(X_{(k+1)\wedge \lambda_{\Lambda}},X_{(k+1)\wedge \lambda_{\Lambda}-M}|^2\Big) \\ \notag
     &\leq\mathbb{E}\Big( |X_{(k+1)\wedge \lambda_{\Lambda}-1}|^2+|g(X_{(k+1)\wedge \lambda_{\Lambda}-1},X_{(k+1)\wedge \lambda_{\Lambda}-1-M})\Delta W_{k\wedge \lambda_{\Lambda}}|^2\Big)+ b_9\Delta-(b_{10}-\lambda_1)\Delta \mathbb{E}(|X_{(k+1)\wedge \lambda_{\Lambda}}|^2)\\
     &\quad~+b_{11}\Delta\mathbb{E}(|X_{(k+1)\wedge \lambda_{\Lambda}-M}|^2)-b_{12}\Delta\mathbb{E}( |X_{(k+1)\wedge \lambda_{\Lambda}}|^{q+1})+b_{13}\mathbb{E}(|X_{(k+1)\wedge \lambda_{\Lambda}-M}|^{q+1}).
    \end{align}
 Furthermore, if necessary, choose a small enough $\Delta>0$ such that $1+\lambda_1 \Delta \leq l_2$. From \eqref{eq--4.4} and the inequality $(1+\lambda_1 \Delta)^M \leq e^{\lambda_1\tau}$ for any $\lambda_1,\Delta >0$, we have
    \begin{align*}
    \notag    \mathbb{E}&\Big(|X_{(k+1)\wedge \lambda_{\Lambda}}|^2+\frac{l_2\Delta}{1+\lambda_1 \Delta} |g(X_{(k+1)\wedge \lambda_{\Lambda}},X_{(k+1)\wedge \lambda_{\Lambda}-M}|^2\Big) \\ \notag
     &\leq (1+\lambda_1\Delta)^{-1}\mathbb{E}\Big(|X_{(k+1)\wedge \lambda_{\Lambda}-1}|^2+\frac{l_2\Delta}{1+\lambda_1 \Delta} |g(X_{(k+1)\wedge \lambda_{\Lambda}-1},X_{(k+1)\wedge \lambda_{\Lambda}-1-M}|^2\Big)\\ \notag
     &\quad~+ (1+\lambda_1 \Delta)^{-1}\Big( b_9\Delta-(b_{10}-\lambda_1)\Delta \mathbb{E}(|X_{(k+1)\wedge \lambda_{\Lambda}}|^2)\\ \notag
     &\quad~+b_{11}\Delta\mathbb{E}(|X_{(k+1)\wedge \lambda_{\Lambda}-M}|^2)-b_{12}\Delta\mathbb{E}( |X_{(k+1)\wedge \lambda_{\Lambda}}|^{q+1})+b_{13}\mathbb{E}(|X_{(k+1)\wedge \lambda_{\Lambda}-M}|^{q+1})\Big)\\ \notag
     &\leq \\ \notag
     &\quad~\vdots \\ \notag
     &\leq (1+\lambda_1\Delta)^{-((k+1)\wedge \lambda_{\Lambda})}\mathbb{E}\Big(|X_{0}|^2+\frac{l_2\Delta}{1+\lambda_1 \Delta} |g(X_{0},X_{-M}|^2\Big)+b_9\Delta\mathbb{E}\Big(\sum\limits_{i=1}^{(k+1)\wedge \lambda_{\Lambda}}(1+\lambda_1 \Delta)^{-i} \Big)\\ \notag
&\quad~-(b_{10}-\lambda_1)\Delta\mathbb{E}\Big( \sum \limits_{i=(k+1)\wedge \lambda_{\Lambda}-M+1}^{(k+1)\wedge \lambda_{\Lambda}}(1+\lambda_1\Delta)^{i-(k+1)\wedge \lambda_{\Lambda}-1}|X_i|^2\Big)\\ \notag
&\quad~+b_{11}\Delta(1+\lambda_1 \Delta)^M\mathbb{E}\Big(\sum\limits_{i=-M+1}^0 (1+\lambda_1\Delta)^{i-(k+1)\wedge \lambda_{\Lambda}-1}|X_i|^2 \Big)\\ \notag
&\quad~-b_{12}\Delta\mathbb{E}\Big( \sum \limits_{i=(k+1)\wedge \lambda_{\Lambda}-M+1}^{(k+1)\wedge \lambda_{\Lambda}}(1+\lambda_1\Delta)^{i-(k+1)\wedge \lambda_{\Lambda}-1}|X_i|^{q+1}\Big)\\
&\quad~+b_{13}\Delta(1+\lambda_1 \Delta)^M\mathbb{E}\Big(\sum\limits_{i=-M+1}^0 (1+\lambda_1\Delta)^{i-(k+1)\wedge \lambda_{\Lambda}-1}|X_i|^{q+1}\Big).
    \end{align*}
    Letting $\Lambda \to \infty $ and applying Fatou's lemma and the dominated convergence theorem, we obtain
    \begin{align*}
         \mathbb{E}&\Big(|X_{k+1}|^2+\frac{l_2\Delta}{1+\lambda_1 \Delta} |g(X_{k+1},X_{k+1-M}|^2\Big) \\ \notag
         &\leq (1+\lambda_1\Delta)^{-(k+1)}\mathbb{E}\Big(|X_{0}|^2+\frac{l_2\Delta}{1+\lambda_1 \Delta} |g(X_{0},X_{-M}|^2\Big)+b_9\Delta\sum\limits_{i=1}^{k+1}(1+\lambda_1 \Delta)^{-i} \\ \notag
&\quad~-(b_{10}-\lambda_1)\Delta \sum \limits_{i=k-M+2}^{k+1}(1+\lambda_1\Delta)^{i-k-2}\mathbb{E}(|X_i|^2)\\ \notag
&\quad~+b_{11}\Delta(1+\lambda_1 \Delta)^M\sum\limits_{i=-M+1}^0 (1+\lambda_1\Delta)^{i-k-2}\mathbb{E}(|X_i|^2) \\ \notag
&\quad~-b_{12}\Delta \sum \limits_{i=k-M+2}^{k+1}(1+\lambda_1\Delta)^{i-k-2}\mathbb{E}(|X_i|^{q+1})\\ \notag
&\quad~+b_{13}\Delta(1+\lambda_1 \Delta)^M\sum\limits_{i=-M+1}^0 (1+\lambda_1\Delta)^{i-k-2}\mathbb{E}(|X_i|^{q+1}).
    \end{align*}
  Since $\lim \limits_{\Delta\to 0}(1+\lambda_1 \Delta)^{-M} =e^{-\lambda_1\tau}$, if necessary, choose a sufficiently small $\Delta_1$ such that for any $\Delta\in(0,\Delta_1]$, we have $(1+\lambda_1 \Delta)^{-M}\geq e^{-0.5\lambda_1\Delta}$. Then, the desired assertion follows.
\end{proof}

The following lemma shows that the numerical solution \eqref{CDSDDE} converges strongly to the underlying solution with an order of $1/2$ over any finite time interval..
\begin{lemma}\label{LM42}
     Let Assumptions \ref{as1}  and \ref{as2} hold, then for any $t > 0$, $\xi\in B(R)$ and $\Delta \in (0,\Delta_1]$, there exists a positive constant $C_T$, independent of $\Delta$ and dependent on $T$, such that for any $t\in [-\tau, T]$,
     \begin{align*}
    \notag   & \mathbb{E}\Big( |x(t_{k+1})-X_{k+1}|^2+l_1\Delta|g(x(t_{k+1}),x(t_{k+1}-\tau))-g(X_{k+1}, X_{k+1-M})|^2 \Big)\\ \notag
    &\leq C_T\big(1+\sup \limits_{t\geq -\tau}\mathbb{E}(|x(t)|^{4q-2})\big) \Delta-(b_1-l_1a_2-\epsilon)\Delta \sum \limits_{i=k-M+2}^{k+1}\mathbb{E}(|x(t_i)-X_i|^2) \\ \notag
   &\quad~-(b_3-l_1a_4)\Delta \sum \limits_{i=k-M+2}^{k+1}\mathbb{E}\big(V(x(t_{i}), X_{i})\big) ,
\end{align*}
where $\epsilon_3$ is a positive constant that satisfies $b_1-l_1a_2-\epsilon -b_2-l_1a_3 >0$.
\end{lemma}
\begin{proof}
    First, for simplicity, setting $e_k= x(t_k)-X_k$, $F_{k}=f(x(t_{k}),x(t_k-\tau))-f(X_k, X_{k-M})$ and $G_{k}=g(x(t_{k}),x(t_k-\tau))-g(X_k, X_{k-M})$. Then, by \eqref{SDDE}, \eqref{DSDDE} and the Young inequality and we have
    \begin{align*}
    \notag    |e_{k+1}|^2&-|e_k|^2+|e_{k+1}-e_k|^2
 \\ \notag
 &= 2\langle e_{k+1}, e_{k+1}-e_k \rangle\\ \notag
    &=2\Big\langle e_{k+1}, F_{k+1}\Delta+G_k \Delta W_k +\int_{t_k}^{t_{k+1}}\big(f(x(s),x(s-\tau))-f(x(t_{k+1}),x(t_{k+1}-\tau))\big)ds \\ \notag
    &\quad +\int_{t_k}^{t_{k+1}}\big(g(x(s),x(s-\tau))-g(x(t_{k}),x(t_{k}-\tau))\big)dW(s) \Big\rangle \\ \notag
    &=2\langle e_{k+1}, F_{k+1}\Delta \rangle+ |e_{k+1}-e_k|^2 \\ \notag
    &\quad+2|G_k\Delta W_k|^2+ 2\Big|\int_{t_k}^{t_{k+1}}\big(g(x(s),x(s-\tau))-g(x(t_{k}),x(t_{k}-\tau))\big)dW(s)\Big|^2 \\ \notag
    &\quad + \epsilon \Delta|e_{k+1}|^2 +\frac{1}{\epsilon \Delta} \Big|\int_{t_k}^{t_{k+1}}\big(f(x(s),x(s-\tau))-f(x(t_{k+1}),x(t_{k+1}-\tau))\big)ds\Big|^2\\ \notag
    &\quad~+2\Big\langle e_k,G_k \Delta W_k+\int_{t_k}^{t_{k+1}}\big(g(x(s),x(s-\tau))-g(x(t_{k}),x(t_{k}-\tau))\big)dW(s) \Big\rangle.
    \end{align*}
   Since $b_1-l_1a_2 -b_2-l_1a_3 >0$, choose $\epsilon _6>0$ small enough such that $b_1-l_1a_2-\epsilon _6-b_2-l_1a_3 >0$, this together with Assumptions \ref{as1} and \ref{as2} implies
    \begin{align}\label{eq-4.2}
 \notag   |e_{k+1}|^2+l_1\Delta|G_{k+1}|^2 &= |e_k|^2+ 2\Delta |G_k|^2 -(b_1-l_1a_2-\epsilon _6)\Delta|e_{k+1}|^2+(b_2+l_1a_3)\Delta|e_{k-M+1}|^2\\ \notag
      &\quad-(b_3-l_1a_4)\Delta V(x(t_{k+1}, X_{k+1}))+l_1a_5 \Delta V(x(t_{k-M+1}, X_{k-M+1})) \\ \notag
    & \quad + 2\Big|\int_{t_k}^{t_{k+1}}\big(g(x(s),x(s-\tau))-g(x(t_{k}),x(t_{k}-\tau))\big)dW(s)\Big|^2 \\ \notag
    &\quad+\frac{1}{\epsilon _6\Delta} \Big|\int_{t_k}^{t_{k+1}}\big(f(x(s),x(s-\tau))-f(x(t_{k+1}),x(t_{k+1}-\tau))\big)ds\Big|^2\\
    &\quad+2\Big\langle e_k,G_k \Delta W_k+\int_{t_k}^{t_{k+1}}\big(g(x(s),x(s-\tau))-g(x(t_{k}),x(t_{k}-\tau))\big)dW(s) \Big\rangle.
    \end{align}
 By Assumption \ref{as1} and the H\"older inequality, we have
 \begin{align}\label{eq-4.3}
 \notag    \mathbb{E}&\Big(\Big|\int_{t_k}^{t_{k+1}}\big(f(x(s),x(s-\tau))-f(x(t_{k+1}),x(t_{k+1}-\tau))\big)ds\Big|^2\Big) \\ \notag
     &\leq \Delta \mathbb{E} \int_{t_k}^{t_{k+1}} \Big|f(x(s),x(s-\tau))-f(x(t_{k+1}),x(t_{k+1}-\tau))\Big|^2ds \\ \notag
    &\leq a_1^2\Delta \int_{t_k}^{t_{k+1}} \mathbb{E}\Big(\big( 1+|x(s)|^{q-1}+|x(s-\tau)|^{q-1}+|x(t_{k+1})|^{q-1}+|x(t_{k+1}-\tau)|^{q-1}\big)\big(|x(s)-x(t_{k+1})|\\ \notag
    &\quad~+|x(s-\tau)-x(t_{k+1}-\tau)|\big)\Big)^2ds \\ \notag
    &\leq  a_1\Delta \int_{t_k}^{t_{k+1}} \Big(\mathbb{E}\big( 1+|x(s)|^{q-1}+|x(s-\tau)|^{q-1}+|x(t_{k+1})|^{q-1}+|x(t_{k+1}-\tau)|^{q-1}\big)^{\frac{4q-2}{q-1}}\Big)^{\frac{2q-2}{4q-2}}\\
&\quad~\times \Big(\mathbb{E}\big(|x(s)-x(t_{k+1})|+|x(s-\tau)-x(t_{k+1}-\tau)|\big)^{\frac{4q-2}{q}}\Big)^{\frac{2q}{4q-2}}.
 \end{align}
 Note that for any $k\in \mathbb{N}$, by Assumption \ref{as1}, the Burkholder–Davis–Gundy inequality, and the elementary inequalities \eqref{EI-1} and \eqref{EI-2}, we have
 \begin{align*}
  \notag   \mathbb{E}\Big(|x(s)-x(t_{k+1})|^{\frac{4q-2}{q}}\Big)&= \mathbb{E}\Big(\Big|\int_s^{t_{k+1}}f(x(s),x(s-\tau))ds+\int_{s}^{t_{k+1}}g(x(s),x(s-\tau))dW(s)\Big|^{\frac{4q-2}{q}}\Big) \\
     &\leq C\big(1+\sup \limits_{t\geq -\tau}\mathbb{E}(|x(t)|^{4q-2})\big)\Delta^{\frac{4q-2}{2q}}.
 \end{align*}
 Thus, combining this with Assumption \ref{as3} and \eqref{eq-4.3} yields
 \begin{align}\label{eq-4.6}
       \mathbb{E}\Big(\Big|\int_{t_k}^{t_{k+1}}\big(f(x(s),x(s-\tau))-f(x(t_{k+1}),x(t_{k+1}-\tau))\big)ds\Big|^2\Big)&\leq C\big(1+\sup \limits_{t\geq -\tau}\mathbb{E}(|x(t)|^{4q-2})\big) \Delta^3.
 \end{align}
 Similarly, we also have
 \begin{align}\label{eq-4.7}
     \mathbb{E}\Big(\Big|\int_{t_k}^{t_{k+1}}\big(g(x(s),x(s-\tau))-g(x(t_{k}),x(t_{k}-\tau))\big)dW(s)\Big|^2\Big)\leq C\big(1+\sup \limits_{t\geq -\tau}\mathbb{E}(|x(t)|^{4q-2})\big) \Delta^2.
 \end{align}
Since $b_1-l_1a_2-\epsilon _6-b_2-l_1a_3>0$ and $b_3-l_1a_4-l_1a_5>0$, by taking the expectation on both sides of \eqref{eq-4.2} and substituting \eqref{eq-4.6} and \eqref{eq-4.7} into the result, we can see that
\begin{align*}
    \notag   & \mathbb{E}\Big( |e_{k+1}|^2+l_1\Delta|G_{k+1}|^2 \Big)\\ \notag
    &\leq \mathbb{E}\Big(|e_k|^2+ l_1\Delta |G_k|^2 \Big)-(b_1-l_1a_2-\epsilon _6)\Delta \mathbb{E}(|e_{k+1}|^2)+(b_2+l_1a_3)\Delta\mathbb{E}(|e_{k-M+1}|^2)\\ \notag
      &\quad-(b_3-l_1a_4)\Delta \mathbb{E}\big(V(x(t_{k+1}), X_{k+1})\big)+l_1a_5 \Delta\mathbb{E}\big( V(x(t_{k-M+1}), X_{k-M+1})\big) +C\big(1+\sup \limits_{t\geq -\tau}\mathbb{E}(|x(t)|^{4q-2})\big) \Delta^2 \\ \notag
       &\leq \\ \notag
    &\quad~\vdots \\ \notag
    &\leq -(b_1-l_1a_2-\epsilon _6)\Delta \sum \limits_{i=k-M+2}^{k+1}\mathbb{E}(|e_i|^2)
   -(b_3-l_1a_4)\Delta \sum \limits_{i=k-M+2}^{k+1}\mathbb{E}\big(V(x(t_{i}), X_{i})\big)  \\ \notag
   &\quad~+CT\big(1+\sup \limits_{t\geq -\tau}\mathbb{E}(|x(t)|^{4q-2})\big) \Delta,
\end{align*}
where we have used the fact that $l_1>2$.
\end{proof}
Next, we formulate the key lemma, which is essential for proving the strong convergence of numerical solutions in the infinite

\begin{lemma}
    Suppose Assumptions \ref{as1} and \ref{as2} hold. There exists a positive constant $\Delta_2$ such that for any $\Delta \in (0,\Delta_1 \wedge \Delta_2]$, any two numerical solutions $\{X_k\}_{k\geq -M}$ and $\{Y_k\}_{k\geq -M}$ of \eqref{DSDDE} with different initial values have the following property
     \begin{align*}
    \notag  \mathbb{E}&\Big(|X_{k+1}-Y_{k+1}|^2 + \frac{l_1\Delta}{1+\lambda_2 \Delta}|g(X_{k+1},X_{k+1-M})-g(Y_{k+1},Y_{k+1-M})|^2\Big) \\ \notag
     &\leq \mathbb{E}\Big(|X_{0}-Y_{0}|^2 + \frac{l_1\Delta}{1+\lambda_2 \Delta}|g(X_{0},X_{-M})-g(Y_{0},Y_{-M})|^2\Big)\Big(1+\lambda_2 \Delta \Big)^{-(k+1)}\\ \notag
     &\quad~-(b_1-l_1a_2-\lambda_2)\Delta \sum \limits_{i=k-M+2}^{k+1}(1+\lambda_2\Delta)^{i-k-2}\mathbb{E}(|X_{i}-Y_{i}|^2) \\ \notag
     &\quad~+b_2(1+\lambda_2 \Delta)^M \Delta \sum \limits_{i=-M+1}^{0}(1+\lambda_2\Delta)^{i-k-2}\mathbb{E}(|X_{i}-Y_{i}|^2) \\ \notag
     &\quad~+ l_1\epsilon_2(1+\lambda_2 \Delta)^M \Delta \sum \limits_{i=-M+1}^{0}(1+\lambda_2\Delta)^{i-k-2}\mathbb{E}\Big(|g_2(X_{i})-g_2(Y_{i})|^2\Big)\\
     &\quad~-(b_3-l_1a_4)\Delta \sum \limits_{i=k-M+2}^{k+1}(1+\lambda_2\Delta)^{i-k-2}\mathbb{E}\big(V(X_i, Y_{i})\big) .
   \end{align*}
    where $\lambda_2$ is a positive constant that satifies $b_1-l_1a_2-(\lambda_2 \vee \epsilon_3) -(b_2+l_1a_3 )e^{\lambda_2 \tau}>0$.
\end{lemma}
\begin{proof}
    First, for simplicity, setting $D_k=X_k-Y_k$, $\overline{F}_{k}=f(X_{k},X_{k-M})-f(Y_k,Y_{k-M})$, and $\overline{G}_{k}=g(X_{k},X_{k-M})-g(Y_k,Y_{k-M})$, then, from \eqref{DSDDE}, we have
    \begin{align*}
     \notag   |D_{k+1}|^2-|D_k|^2+|D_{k+1}-D_k|^2 &=2\Big\langle D_{k+1},D_{k+1}-D_k  \rangle \\ \notag
     &=2\Big \langle D_{k+1}, \overline{F}_{k+1}\Delta+\overline{G}_k\Delta W_k \Big\rangle \\ \notag
     &\leq 2\Big\langle D_{k+1}, \overline{F}_{k+1}\Delta \Big \rangle + |D_{k+1}-D_k|^2 + |\overline{G}_k\Delta W_k|^2+2\Big\langle D_{k}, \overline{G}_{k}\Delta W_k \Big \rangle.
    \end{align*}
    Since $b_1-l_1a_2 -\epsilon_3-b_2-l_1a_3 >0$ and $b_3-l_1a_4-l_1a_5 >0$, choose $\lambda_2 >0$ small enough such that $b_1-l_1a_2-(\lambda_2 \vee \epsilon_3) -(b_2+l_1a_3 )e^{\lambda_2 \tau}>0$ and $b_3-l_1a_4-l_1a_5e^{\lambda_2 \tau} >0$, this together with Assumptions \ref{as1} and \ref{as2} shows that
     \begin{align}\label{eq-4.8}
   \notag  (1+\lambda_2\Delta)&  |D_{k+1}|^2 + l_1\Delta|\overline{G}_{k+1}|^2\\ \notag
   &\leq |D_k|^2+2|\overline{G}_k \Delta W_k|^2-(b_1-l_1a_2-\lambda_2)\Delta|D_{k+1}|^2+b_2\Delta|D_{k-M+1}|^2\\
     &\quad -(b_3-l_1a_4)\Delta V(X_{k+1}, Y_{k+1})+ l_1\epsilon_2\Delta|g_2(X_{k-M+1})-g_2(Y_{k-M+1})|^2+2\Big\langle D_{k}, \overline{G}_{k}\Delta W_k \Big \rangle.
   \end{align}
   Furthermore, if necessary, choose a small enough $\Delta$ such that $2 \leq l_1/(1+\lambda_2\Delta)$. Then, by \eqref{eq-4.8} and the inequality $(1+\lambda_2 \Delta)^M \leq e^{\lambda_2 \tau}$ for any $\lambda_2, \Delta \geq 0$, we have
    \begin{align*}
    \notag  \mathbb{E}&\Big(|D_{k+1}|^2 + \frac{l_1\Delta}{1+\lambda_2 \Delta}|\overline{G}_{k+1}|^2\Big) \\ \notag
       &\leq  \Bigg(1+\lambda_2 \Delta\Bigg)^{-1} \Bigg(\mathbb{E}\Big(|D_{k}|^2 + \frac{l_1\Delta}{1+\lambda_2 \Delta}|\overline{G}_{k}|^2\Big)-(b_1-l_1a_2-\lambda_2)\Delta)\mathbb{E}(|D_{k+1}|^2)\\ \notag
     &\quad~ +b_2\Delta \mathbb{E}(|D_{k-M+1}|^2)-(b_3-l_1a_4))\Delta\mathbb{E}\Big(V(X_{k+1}, Y_{k+1})\Big)+ l_1\epsilon_2\Delta\mathbb{E}\Big(|g_2(X_{k-M+1})-g_2(Y_{k-M+1})|^2\Big)\Bigg) \\ \notag
     &\leq \\ \notag
     &\quad~\vdots \\ \notag
     &\leq \mathbb{E}\Big(|D_{0}|^2 + \frac{l_1\Delta}{1+\lambda_2 \Delta}|\overline{G}_{0}|^2\Big)\Big(1+\lambda_2 \Delta \Big)^{-(k+1)}-(b_1-l_1a_2-\lambda_2)\Delta \sum \limits_{i=k-M+2}^{k+1}(1+\lambda_2\Delta)^{i-k-2}\mathbb{E}(|D_i|^2) \\ \notag
     &\quad~+b_2(1+\lambda_2 \Delta)^M \Delta \sum \limits_{i=-M+1}^{0}(1+\lambda_2\Delta)^{i-k-2}\mathbb{E}(|D_i|^2) \\ \notag
     &\quad~+ l_1\epsilon_2(1+\lambda_2 \Delta)^M \Delta \sum \limits_{i=-M+1}^{0}(1+\lambda_2\Delta)^{i-k-2}\mathbb{E}\Big(|g_2(X_{i})-g_2(Y_{i})|^2\Big)\\
     &\quad~-(b_3-l_1a_4)\Delta \sum \limits_{i=k-M+2}^{k+1}(1+\lambda_2\Delta)^{i-k-2}\mathbb{E}\big(V(X_i, Y_{i})\big) .
   \end{align*}
\end{proof}

We now state one of the central results of this work.

\begin{theorem}\label{T--2}
 Suppose Assumptions \ref{as1}, \ref{as2} and \ref{as3} hold, then the numerical solution of \eqref{DSDDE} converges strongly to the underlying solution of \eqref{SDDE} in the infinite horizon, i.e., for any $k\geq -M$ and $\xi \in B(R)$
 \begin{displaymath}
       \mathbb{E}\Big(|x(t_k; 0, \xi)-X_{k}^{0,\xi}|^2\Big) \leq C \Delta,
 \end{displaymath}
  where $C$ is a constant independent of $t$ but dependent on the initial value $\xi$.
\end{theorem}
\begin{proof}
    First, fixed $T=2\log2/\lambda_2+1$. Then, for any $t_k\in [0,T]$, it follows from Lemma \ref{LM42} that
    \begin{align}\label{eq--4.8}
     \notag   \mathbb{E}&\Big(|x(t_k;0,\xi)-X_k^{0,\xi}|^2+\frac{1\Delta}{1+\lambda_2 \Delta} |g(x(t_k;0,\xi),x(t_k-\tau;0,\xi))-g(X_{k}^{0,\xi},X_{k-M}^{0,\xi})|^2\Big) \\ \notag
     &\leq CT\Delta\big(1+\sup \limits_{t\geq -\tau}\mathbb{E}(|x(t;0,\xi)|^{4q-2})\big)  -(b_1-l_1a_2-\epsilon _6)\Delta \sum \limits_{i=k-M+1}^{k}\mathbb{E}(|x(t_i;0,\xi)-X_i^{0,\xi}|^2)
   \\
  &\quad~ -(b_3-l_1a_4)\Delta \sum \limits_{i=k-M+1}^{k}\mathbb{E}\big(V(x(t_{i};0,\xi), X_{i}^{0,\xi})\big).
    \end{align}
    By the Young inequality, we can see that for any $t_k \in [T,2T]$
    \begin{align}\label{eq-4.9}
        \notag   \mathbb{E}&\Big(|x(t_k;0,\xi)-X_k^{0,\xi}|^2+\frac{1\Delta}{1+\lambda_2 \Delta} |g(x(t_k;0,\xi),x(t_k-\tau;0,\xi))-g(X_{k}^{0,\xi},X_{k-M}^{0,\xi})|^2\Big) \\ \notag
        &\leq 2\mathbb{E}\Big(|x(t_k;0,\xi)-X_k^{N,x_{T}^{0,\xi}}|^2+\frac{1\Delta}{1+\lambda_2 \Delta} |g(x(t_k;0,\xi),x(t_k-\tau;0,\xi))-g(X_k^{N,x_{T}^{0,\xi}},X_{k-M}^{N,x_{T}^{0,\xi}})|^2\Big) \\ \notag
        &\quad~+2\mathbb{E}\Big(|X_k^{N,x_{T}^{0,\xi}}-X_k^{0,\xi}|^2+\frac{1\Delta}{1+\lambda_2 \Delta} |g(X_k^{N,x_{T}^{0,\xi}},X_{k-M}^{N,x_{T}^{0,\xi}})|^2-g(X_{k}^{0,\xi},X_{k-M}^{0,\xi})|^2\Big) \\ \notag
        &= 2\mathbb{E}\Big(|x(t_k;T,x_T^{0,\xi})-X_k^{N,x_{T}^{0,\xi}}|^2+\frac{1\Delta}{1+\lambda_2 \Delta} |g(x(t_k;T,x_T^{0,\xi}),x(t_k-\tau;T,x_T^{0,\xi}))-g(X_k^{N,x_{T}^{0,\xi}},X_{k-M}^{N,x_{T}^{0,\xi}})|^2\Big) \\ \notag
        &\quad~+2\mathbb{E}\Big(|X_k^{N,x_{T}^{0,\xi}}-X_k^{N,X_N^{0,\xi}}|^2+\frac{1\Delta}{1+\lambda_2 \Delta} |g(X_k^{N,x_{T}^{0,\xi}},X_{k-M}^{N,x_{T}^{0,\xi}})|^2-g(X_{k}^{N,X_N^{0,\xi}},X_{k-M}^{N,X_N^{0,\xi}})|^2\Big) \\ \notag
        &\leq 2\Bigg(CT\Delta\big(1+\sup \limits_{t\geq -\tau}\mathbb{E}(|x(t;T,x_{T}^{0,\xi})|^{4q-2})\big)  -(b_1-l_1a_2-\epsilon _6)\Delta \sum \limits_{i=k-M+1}^{k}\mathbb{E}(|x(t_i;T,x_{T}^{0,\xi})-X_i^{N,x_{T}^{0,\xi}}|^2)
   \\ \notag
  &\quad~ -(b_3-l_1a_4)\Delta \sum \limits_{i=k-M+1}^{k}\mathbb{E}\big(V(x(t_{i};T,x_{T}^{0,\xi}), X_{i}^{N,x_{T}^{0,\xi}})\big) \Bigg)+2\Bigg(\mathbb{E}\Big(|x(T;0,\xi)-X_{N}^{0,\xi}|^2 \\ \notag
  &\quad~+\frac{l_1\Delta}{1+\lambda_2 \Delta}|g(x(T;0,\xi),x(T-\tau;0,\xi))-g(X_{N}^{0,\xi},X_{N-M}^{0,\xi})|^2\Big)\Big(1+\lambda_2 \Delta \Big)^{-(k-N)} \\ \notag
  &\quad~-(b_1-l_1a_2-\lambda_2)\Delta \sum \limits_{i=k-M+1}^{k}(1+\lambda_2\Delta)^{i-k-1}\mathbb{E}(|X_{i}^{N,x_{T}^{0,\xi}}-X_{i}^{N,X_N^{0,\xi}}|^2) \\ \notag
   &\quad~+b_2(1+\lambda_2 \Delta)^M \Delta \sum \limits_{i=N-M+1}^{N}(1+\lambda_2\Delta)^{i-k-1}\mathbb{E}(|X_{i}^{N,x_{T}^{0,\xi}}-X_{i}^{N,X_N^{0,\xi}}|^2) \\ \notag
    &\quad~-(b_3-l_1a_4)\Delta \sum \limits_{i=k-M+1}^{k}(1+\lambda_2\Delta)^{i-k-1}\mathbb{E}\big(V(X_{i}^{N,x_{T}^{0,\xi}},X_{i}^{N,X_N^{0,\xi}})\big)\\
    &\quad~+l_1\epsilon_2(1+\lambda_2 \Delta)^M \Delta  \sum \limits_{i=N-M+1}^{N}(1+\lambda_2\Delta)^{i-k-1}\mathbb{E}\big(\big|g_2(X_{i}^{N,x_{T}^{0,\xi}})-g_2(X_{i}^{N,X_N^{0,\xi}})\big|^2\big)\Bigg)
    \end{align}
    From \eqref{eq--4.8}, \eqref{eq-4.9} and Assumption \ref{as1}, we note that
    \begin{align*}
     \notag   \Bigg(&-(b_1-l_1a_2-\epsilon _6)\Delta \sum \limits_{i=N-M+1}^{N}\mathbb{E}(|x(t_i;0,\xi)-X_i^{0,\xi}|^2)
  -(b_3-l_1a_4)\Delta \sum \limits_{i=N-M+1}^{N}\mathbb{E}\big(V(x(t_{i};0,\xi), X_{i}^{0,\xi})\big)\\ \notag
  &\quad +b_2(1+\lambda_2 \Delta)^M \Delta \sum \limits_{i=N-M+1}^{N}(1+\lambda_2\Delta)^{i-k-1}\mathbb{E}(|X_{i}^{N,x_{T}^{0,\xi}}-X_{i}^{N,X_N^{0,\xi}}|^2)\\ \notag
  &\quad+l_1\epsilon_2(1+\lambda_2 \Delta)^M \Delta  \sum \limits_{i=N-M+1}^{N}\mathbb{E}\big(\big|g_2(X_{i}^{N,x_{T}^{0,\xi}})-g_2(X_{i}^{N,X_N^{0,\xi}})\big|^2\big)\Bigg)\Big(1+\lambda_2 \Delta \Big)^{-(k-N)} \\ \notag
&=\Bigg(-(b_1-l_1a_2-\epsilon _6-(b_2+l_1a_3)(1+\lambda_2 \Delta)^M)\Delta \sum \limits_{i=N-M+1}^{N}\mathbb{E}(|x(t_i;0,\xi)-X_i^{0,\xi}|^2) \\ \notag
 &~\quad\quad -(b_3-l_1a_4-l_1a_5(1+\lambda_2 \Delta)^M)\Delta \sum \limits_{i=N-M+1}^{N}\mathbb{E}\big(V(x(t_{i};0,\xi), X_{i}^{0,\xi})\big)\Bigg) \\ \notag
 &\leq \Bigg(-(b_1-l_1a_2-(\epsilon _6\vee \lambda_2)-(b_2+l_1a_3)e^{\lambda_2 \tau})\Delta \sum \limits_{i=N-M+1}^{N}\mathbb{E}(|x(t_i;0,\xi)-X_i^{0,\xi}|^2) \\
 &~\quad\quad -(b_3-l_1a_4-l_1a_5e^{\lambda_2 \tau})\Delta \sum \limits_{i=N-M+1}^{N}\mathbb{E}\big(V(x(t_{i};0,\xi), X_{i}^{0,\xi})\big)\Bigg) \leq 0.
    \end{align*}
    Thus, we have
    \begin{align*}
          \notag   \mathbb{E}&\Big(|x(t_k;0,\xi)-X_k^{0,\xi}|^2+\frac{1\Delta}{1+\lambda_2 \Delta} |g(x(t_k;0,\xi),x(t_k-\tau;0,\xi))-g(X_{k}^{0,\xi},X_{k-M}^{0,\xi})|^2\Big) \\ \notag
          &\leq CT\Delta\big(1+\sup \limits_{t\geq -\tau}\mathbb{E}(|x(t;0,\xi)|^{4q-2})\big) \big(2+2(1+\lambda_2 \Delta)^{-(k-N)}\big)\\ \notag
          &\quad~+2\Bigg( -(b_1-l_1a_2-\epsilon _6)\Delta \sum \limits_{i=k-M+1}^{k}\mathbb{E}(|x(t_i;T,x_{T}^{0,\xi})-X_i^{N,x_{T}^{0,\xi}}|^2)
   \\ \notag
  &\quad~ -(b_3-l_1a_4)\Delta \sum \limits_{i=k-M+1}^{k}\mathbb{E}\big(V(x(t_{i};T,x_{T}^{0,\xi}), X_{i}^{N,x_{T}^{0,\xi}})\big) \\ \notag
  &\quad~-(b_1-l_1a_2-\lambda_2)\Delta \sum \limits_{i=k-M+1}^{k}(1+\lambda_2\Delta)^{i-k-1}\mathbb{E}(|X_{i}^{N,x_{T}^{0,\xi}}-X_{i}^{N,X_N^{0,\xi}}|^2) \\ \notag
   &\quad~-(b_3-l_1a_4)\Delta \sum \limits_{i=k-M+1}^{k}(1+\lambda_2\Delta)^{i-k-1}\mathbb{E}\big(V(X_{i}^{N,x_{T}^{0,\xi}},X_{i}^{N,X_N^{0,\xi}})\big)\Bigg).
    \end{align*}
    Similarly, for any $t_k \in [2T,3T]$, we have
     \begin{align}\label{eq-4.12}
        \notag   \mathbb{E}&\Big(|x(t_k;0,\xi)-X_k^{0,\xi}|^2+\frac{1\Delta}{1+\lambda_2 \Delta} |g(x(t_k;0,\xi),x(t_k-\tau;0,\xi))-g(X_{k}^{0,\xi},X_{k-M}^{0,\xi})|^2\Big) \\ \notag
        &\leq 2\Bigg(CT\Delta\big(1+\sup \limits_{t\geq -\tau}\mathbb{E}(|x(t;2T,x_{2T}^{0,\xi})|^{4q-2})\big)  -(b_1-l_1a_2-\epsilon _6)\Delta \sum \limits_{i=k-M+1}^{k}\mathbb{E}(|x(t_i;2T,x_{2T}^{0,\xi})-X_i^{2N,x_{2T}^{0,\xi}}|^2)
   \\ \notag
  &\quad~ -(b_3-l_1a_4)\Delta \sum \limits_{i=k-M+1}^{k}\mathbb{E}\big(V(x(t_{i};2T,x_{2T}^{0,\xi}), X_{i}^{2N,x_{2T}^{0,\xi}})\big) \Bigg)+2\Bigg(\mathbb{E}\Big(|x(2T;0,\xi)-X_{2N}^{0,\xi}|^2 \\ \notag
  &\quad~+\frac{l_1\Delta}{1+\lambda_2 \Delta}|g(x(2T;0,\xi),x(2T-\tau;0,\xi))-g(X_{2N}^{0,\xi},X_{2N-M}^{0,\xi})|^2\Big)\Big(1+\lambda_2 \Delta \Big)^{-(k-2N)} \\ \notag
  &\quad~-(b_1-l_1a_2-\lambda_2)\Delta \sum \limits_{i=k-M+1}^{k}(1+\lambda_2\Delta)^{i-k-1}\mathbb{E}(|X_{i}^{2N,x_{2T}^{0,\xi}}-X_{i}^{2N,X_{2N}^{0,\xi}}|^2) \\ \notag
   &\quad~+b_2(1+\lambda_2 \Delta)^M \Delta \sum \limits_{i=2N-M+1}^{2N}(1+\lambda_2\Delta)^{i-k-1}\mathbb{E}(|X_{i}^{2N,x_{2T}^{0,\xi}}-X_{i}^{2N,X_{2N}^{0,\xi}}|^2) \\ \notag
    &\quad~-(b_3-l_1a_4)\Delta \sum \limits_{i=k-M+1}^{k}(1+\lambda_2\Delta)^{i-k-1}\mathbb{E}\big(V(X_{i}^{2N,x_{2T}^{0,\xi}},X_{i}^{2N,X_{2N}^{0,\xi}})\big)\\
    &\quad~+l_1\epsilon_2(1+\lambda_2 \Delta)^M \Delta  \sum \limits_{i=2N-M+1}^{2N}(1+\lambda_2\Delta)^{i-k-1}\mathbb{E}\big(\big|g_2(X_{i}^{2N,x_{2T}^{0,\xi}})-g_2(X_{i}^{2N,X_{2N}^{0,\xi}})\big|^2\big)\Bigg).
    \end{align}
  By Assumption \ref{as1}, we note that
  \begin{align}\label{eq-4.13}
   \notag  \Bigg(& l_1\epsilon_2(1+\lambda_2 \Delta)^M \Delta  \sum \limits_{i=2N-M+1}^{2N}\mathbb{E}\big((1+\lambda_2\Delta)^{i-2N-1}\big|g_2(X_{i}^{2N,x_{2T}^{0,\xi}})-g_2(X_{i}^{2N,X_{2N}^{0,\xi}})\big|^2\big) \\ \notag
     & -2\Big((b_3-l_1a_4)\Delta \sum \limits_{i=2N-M+1}^{2N}\mathbb{E}\big(V(x(t_{i};T,x_{T}^{0,\xi}), X_{i}^{N,x_{T}^{0,\xi}})\big)\\ \notag
     &+(b_3-l_1a_4)\Delta \sum \limits_{i=2N-M+1}^{2N}(1+\lambda_2\Delta)^{i-2N-1}\mathbb{E}\big(V(X_{i}^{N,x_{T}^{0,\xi}},X_{i}^{N,X_N^{0,\xi}})\big)\Big)\Bigg) \\ \notag
     &\leq \Bigg(2 l_1\epsilon_2(1+\lambda_2 \Delta)^M \Delta  \sum \limits_{i=2N-M+1}^{2N}(1+\lambda_2\Delta)^{i-2N-1}\Big(\mathbb{E}\big(\big|g_2(X_{i}^{2N,x_{2T}^{0,\xi}})-g_2(X_{i}^{N,x_{T}^{0,\xi}})\big|^2\big)\\ \notag
     & \quad\quad+\mathbb{E}\big(\big|g_2(X_{i}^{N,x_{T}^{0,\xi}})-g_2(X_{i}^{2N,X_{2N}^{0,\xi}})\big|^2\big) \Big)-2\Big((b_3-l_1a_4)\Delta \sum \limits_{i=2N-M+1}^{2N}\mathbb{E}\big(V(x(t_{i};T,x_{T}^{0,\xi}), X_{i}^{N,x_{T}^{0,\xi}})\big)\\ \notag
     &\quad\quad+(b_3-l_1a_4)\Delta \sum \limits_{i=2N-M+1}^{2N}(1+\lambda_2\Delta)^{i-2N-1}\mathbb{E}\big(V(X_{i}^{N,x_{T}^{0,\xi}},X_{i}^{N,X_N^{0,\xi}})\big)\Big)\Bigg) \\ \notag
     &\leq -2\Big((b_3-l_1a_4-l_1a_5(1+\lambda_2 \Delta)^M )\Delta \sum \limits_{i=2N-M+1}^{2N}\mathbb{E}\big(V(x(t_{i};T,x_{T}^{0,\xi}), X_{i}^{N,x_{T}^{0,\xi}})\big)\\ \notag
     &\quad\quad+(b_3-l_1a_4-l_1a_5(1+\lambda_2 \Delta)^M )\Delta \sum \limits_{i=2N-M+1}^{2N}(1+\lambda_2\Delta)^{i-2N-1}\mathbb{E}\big(V(X_{i}^{N,x_{T}^{0,\xi}},X_{i}^{N,X_N^{0,\xi}})\big)\Big) \\ \notag
     &\quad\quad+2l_1a_3(1+\lambda_2 \Delta)^M \Delta  \sum \limits_{i=2N-M+1}^{2N}(1+\lambda_2\Delta)^{i-2N-1}\Big(\mathbb{E}\big(\big|X_{i}^{2N,x_{2T}^{0,\xi}}-X_{i}^{N,x_{T}^{0,\xi}}\big|^2\big)\\ \notag
     &\quad\quad+\mathbb{E}\big(\big|X_{i}^{N,x_{T}^{0,\xi}}-X_{i}^{2N,X_{2N}^{0,\xi}}\big|^2\big) \Big) \\
     &\quad\leq2 l_1a_3(1+\lambda_2 \Delta)^M \Delta  \sum \limits_{i=2N-M+1}^{2N}(1+\lambda_2\Delta)^{i-2N-1}\Big(\mathbb{E}\big(\big|X_{i}^{2N,x_{2T}^{0,\xi}}-X_{i}^{N,x_{T}^{0,\xi}}\big|^2\big)+\mathbb{E}\big(\big|X_{i}^{N,x_{T}^{0,\xi}}-X_{i}^{2N,X_{2N}^{0,\xi}}\big|^2\big) \Big)
  \end{align}
  and
  \begin{align}\label{eq-4.14}
   \notag    &b_2(1+\lambda_2 \Delta)^M \Delta \sum \limits_{i=2N-M+1}^{2N}(1+\lambda_2\Delta)^{i-k-1}\mathbb{E}(|X_{i}^{2N,x_{2T}^{0,\xi}}-X_{i}^{2N,X_{2N}^{0,\xi}}|^2)\\ \notag
      &\quad+2l_1a_3(1+\lambda_2 \Delta)^M \Delta  \sum \limits_{i=2N-M+1}^{2N}(1+\lambda_2\Delta)^{i-2N-1}\Big(\mathbb{E}\big(\big|X_{i}^{2N,x_{2T}^{0,\xi}}-X_{i}^{N,x_{T}^{0,\xi}}\big|^2\big)+\mathbb{E}\big(\big|X_{i}^{N,x_{T}^{0,\xi}}-X_{i}^{2N,X_{2N}^{0,\xi}}\big|^2\big) \Big) \\ \notag
       &\quad~-2\Bigg((b_1-l_1a_2-\epsilon _6)\Delta \sum \limits_{i=2N-M+1}^{2N}\mathbb{E}(|x(t_i;T,x_{T}^{0,\xi})-X_i^{N,x_{T}^{0,\xi}}|^2)
   \\ \notag
  &\quad~+(b_1-l_1a_2-\lambda_2)\Delta \sum \limits_{i=2N-M+1}^{2N}(1+\lambda_2\Delta)^{i-k-1}\mathbb{E}(|X_{i}^{N,x_{T}^{0,\xi}}-X_{i}^{N,X_N^{0,\xi}}|^2) \Bigg)\\ \notag
  &\quad\quad\leq -2\Bigg((b_1-l_1a_2-\epsilon _6-(b_2+l_1a_3)(1+\lambda_2 \Delta)^M)\Delta \sum \limits_{i=2N-M+1}^{2N}\mathbb{E}(|x(t_i;T,x_{T}^{0,\xi})-X_i^{N,x_{T}^{0,\xi}}|^2)
   \\
  &\quad\quad~~+(b_1-l_1a_2-\lambda_2-(b_2+l_1a_3)(1+\lambda_2 \Delta)^M)\Delta \sum \limits_{i=2N-M+1}^{2N}(1+\lambda_2\Delta)^{i-k-1}\mathbb{E}(|X_{i}^{N,x_{T}^{0,\xi}}-X_{i}^{N,X_N^{0,\xi}}|^2) \Bigg)  \leq 0.
  \end{align}
  Thus, it follows from \eqref{eq-4.12}, \eqref{eq-4.13} and \eqref{eq-4.14} that for any $t_k\in [2T,3T]$
   \begin{align*}
          \notag   \mathbb{E}&\Big(|x(t_k;0,\xi)-X_k^{0,\xi}|^2+\frac{1\Delta}{1+\lambda_2 \Delta} |g(x(t_k;0,\xi),x(t_k-\tau;0,\xi))-g(X_{k}^{0,\xi},X_{k-M}^{0,\xi})|^2\Big) \\ \notag
          &\leq CT\Delta\big(1+\sup \limits_{t\geq -\tau}\mathbb{E}(|x(t;0,\xi)|^{4q-2})\big) \big(2+2^2(1+\lambda_2 \Delta)^{-(k-2N)}++2^3(1+\lambda_2 \Delta)^{-(k-N)}\big)\\ \notag
          &\quad~+2\Bigg( -(b_1-l_1a_2-\epsilon _6)\Delta \sum \limits_{i=k-M+1}^{k}\mathbb{E}(|x(t_i;2T,x_{2T}^{0,\xi})-X_i^{2N,x_{2T}^{0,\xi}}|^2)
   \\ \notag
  &\quad~ -(b_3-l_1a_4)\Delta \sum \limits_{i=k-M+1}^{k}\mathbb{E}\big(V(x(t_{i};2T,x_{2T}^{0,\xi}), X_{i}^{2N,x_{2T}^{0,\xi}})\big) \\ \notag
  &\quad~-(b_1-l_1a_2-\lambda_2)\Delta \sum \limits_{i=k-M+1}^{k}(1+\lambda_2\Delta)^{i-k-1}\mathbb{E}(|X_{i}^{2N,x_{2T}^{0,\xi}}-X_{i}^{2N,X_{2N}^{0,\xi}}|^2) \\ \notag
   &\quad~-(b_3-l_1a_4)\Delta \sum \limits_{i=k-M+1}^{k}(1+\lambda_2\Delta)^{i-k-1}\mathbb{E}\big(V(X_{i}^{2N,x_{2T}^{0,\xi}},X_{i}^{2N,X_{2N}^{0,\xi}})\big)\Bigg).
    \end{align*}
Continue this process and it is not difficult to see that for any $t_k \in [jT,(j+1)T]$ with $j\in \mathbb{N}$
\begin{align*}
      \notag   \mathbb{E}&\Big(|x(t_k;0,\xi)-X_k^{0,\xi}|^2+\frac{1\Delta}{1+\lambda_2 \Delta} |g(x(t_k;0,\xi),x(t_k-\tau;0,\xi))-g(X_{k}^{0,\xi},X_{k-M}^{0,\xi})|^2\Big) \\ \notag
          &\leq 2Ce^{\lambda_2 T}T\Delta\big(1+\sup \limits_{t\geq -\tau}\mathbb{E}(|x(t;0,\xi)|^{4q-2})\big)\sum\limits_{i=0}^{j} \Big(2(1+\lambda_2 \Delta)^{-N}\Big)^i \\ \notag
          &\quad~+2\Bigg( -(b_1-l_1a_2-\epsilon _6)\Delta \sum \limits_{i=k-M+1}^{k}\mathbb{E}(|x(t_i;2T,x_{2T}^{0,\xi})-X_i^{2N,x_{2T}^{0,\xi}}|^2)
   \\ \notag
  &\quad~ -(b_3-l_1a_4)\Delta \sum \limits_{i=k-M+1}^{k}\mathbb{E}\big(V(x(t_{i};2T,x_{2T}^{0,\xi}), X_{i}^{2N,x_{2T}^{0,\xi}})\big) \\ \notag
  &\quad~-(b_1-l_1a_2-\lambda_2)\Delta \sum \limits_{i=k-M+1}^{k}(1+\lambda_2\Delta)^{i-k-1}\mathbb{E}(|X_{i}^{2N,x_{2T}^{0,\xi}}-X_{i}^{2N,X_{2N}^{0,\xi}}|^2) \\ \notag
   &\quad~-(b_3-l_1a_4)\Delta \sum \limits_{i=k-M+1}^{k}(1+\lambda_2\Delta)^{i-k-1}\mathbb{E}\big(V(X_{i}^{2N,x_{2T}^{0,\xi}},X_{i}^{2N,X_{2N}^{0,\xi}})\big)\Bigg).
\end{align*}
If necessary, choose a sufficiently small $\Delta$ such that $(1+\lambda_2 \Delta)^{-N} \leq e^{-\frac{\lambda_2 T}{2}}$. Note that $\hat{\theta}:=2e^{-\frac{\lambda_2 T}{2}} <1$, thus, $\sum_{i=0}^{\infty}(2(1+\lambda_2 \Delta)^{-N})^i \leq 1/(1-\hat{\theta})$. Then, we have
\begin{align*}
   \notag   \mathbb{E}&\Big(|x(t_k;0,\xi)-X_k^{0,\xi}|^2+\frac{1\Delta}{1+\lambda_2 \Delta} |g(x(t_k;0,\xi),x(t_k-\tau;0,\xi))-g(X_{k}^{0,\xi},X_{k-M}^{0,\xi})|^2\Big) \\ \notag
    &\leq \frac{ 2Ce^{\lambda_2 T}T}{1-\hat{\theta}}\big(1+\sup \limits_{t\geq -\tau}\mathbb{E}(|x(t;0,\xi)|^{4q-2})\big) \Delta.
\end{align*}
By the arbitrariness of $j$ and Lemma 1, the desired result follows.
\end{proof}

Based on the above lemmas, we take a further step to analyze the uniform boundedness in probability of the norm of the numerical segment process.
\begin{lemma}\label{LM510}
     Suppose Assumptions \ref{as1} and \ref{as2} hold. Then for any $T \geq 0$ and $ \epsilon_7>0$, there exists a positive constant $H$ such that for any $k \in \mathbb{N}$ and $\xi \in B(R)$
     \begin{align*}
         \mathbb{P}\{\|X_{t_k}^{0,\xi}\|\leq  H, \forall s \in [t_k,t_k+T]\}  \geq 1-\epsilon_7.
     \end{align*}
\end{lemma}
\begin{proof}
   For any $k,i \in \mathbb{N}$, define
     \begin{align*}
     \overline{\lambda}_k^{\Lambda}=\inf\{i\geq k: |X_{i}>\Lambda|\}.
    \end{align*}
      Then, similar to  Lemma \ref{LM57}, for any $k\in\mathbb{N}$ and $T>0$, we have
\begin{align}\label{eq-4.27}
 \notag  \mathbb{E}&\Big(|X_{(k+N) \wedge \overline{\lambda}_k^{\Lambda}}|^2 +\frac{l_2\Delta}{1+\lambda_1 \Delta}|g(X_{(k+N) \wedge \overline{\lambda}_k^{\Lambda}}, X_{((k+N) \wedge \overline{\lambda}_k^{\Lambda})-M})|^2 \Big)\\ \notag
   &\leq \mathbb{E}\Big(|X_k|^2+\frac{l_2\Delta}{1+\lambda_1 \Delta}|g(X_k,X_{k-M})|^2\Big) \Big(1+\lambda_1 \Delta\Big)^{-(N\wedge (\overline{\lambda}_{k}^{\Lambda}-k))}\\ \notag
   &\quad~+b_9\Delta \mathbb{E}\Big(\sum \limits_{j=k+1}^{(k+N) \wedge \overline{\lambda}_k^{\Lambda}}(1+\lambda_1 \Delta)^{-((k+N) \wedge \overline{\lambda}_k^{\Lambda}+1-j)}\Big) \\ \notag
    &\quad~+b_{11}\Delta\mathbb{E}\Big( \sum\limits_{j=k-M+1}^{k}(1+\lambda_1 \Delta)^{j-((k+N) \wedge \overline{\lambda}_k^{\Lambda}+1)}|X_j|^{2}\Big) \\ \notag
   &\quad~+b_{13}\Delta\mathbb{E}\Big( \sum\limits_{j=k-M+1}^{k}(1+\lambda_1 \Delta)^{j-((k+N) \wedge \overline{\lambda}_k^{\Lambda}+1)}|X_j|^{q+1}\Big) \\ \notag
   &\leq C\Big(1+ \sup \limits_{i\geq -M} \mathbb{E}\big(|X_i|^2+\frac{l_2\Delta}{1+\lambda_1 \Delta}|g(X_i,X_{i-M})|^2\big)+ \sup\limits_{k\geq0}\Delta \sum\limits_{j=k-M+1}^{k} \mathbb{E}(|X_j|^{q+1}) \Big)e^{\lambda T} \\
   &\leq Re^{\lambda T},
\end{align}
where $R$ is a constant that depends on $\xi$. For any $\epsilon _7 >0$, choose a positive constant $K_1$ large enough such that
\begin{align}\label{eq-4.28}
    Re^{\lambda T } \leq \frac{\epsilon _7 K_1^2}{2}.
\end{align}
By the Chebyshev inequality, it follows from \eqref{eq-4.27} and \eqref{eq-4.28} that
\begin{align*}
    \mathbb{P}\big\{\overline{\lambda}_k^{K_1} &\leq k+N\big\}\leq \mathbb{P}\big\{X_{(k+N)\wedge  \overline{\lambda}_k^{K_1}} >K_1\big\} \\ \notag
    &\leq \frac{\mathbb{E}\big(|X_{(k+N)\wedge  \overline{\lambda}_k^{K_1}} |^2\big)}{K_1^2} \\ \notag
    &\leq \frac{Re^{\lambda T}}{K_1^2}\leq \frac{\epsilon _7}{2},
\end{align*}
which implies that
\begin{align}\label{eq-4.40}
    \mathbb{P}\big\{X(t;0,\xi)\geq K_1, \forall t \in [t_k,t_k+T]\big\} \leq \frac{\epsilon _7}{2}
\end{align}
for any $t \geq 0$. By Lemma \ref{LM3.1}, we know that for any $\epsilon _7 >0$, there exists a constant $K_2$ such that
\begin{align}\label{eq-4.41}
    \mathbb{P}\big\{X(t;0,\xi)\geq K_2, \forall t \in [-\tau,0]\big\} \leq \frac{\epsilon _7}{2}.
\end{align}
Setting $H=K_1 \vee K_2$, and combining \eqref{eq-4.40} and \eqref{eq-4.41} together yields
\begin{align}\label{eq-4.42}
      \mathbb{P}\big\{X(t;0,\xi)\leq H, \forall t \in [t_k,t_k+T]\big\} \geq 1-\epsilon _7
\end{align}
for any $t \geq -\tau$. Furthermore, we have
\begin{align*}
    \mathbb{P}\big\{\|X_{t}^{0,\xi}\|\leq H, \forall t \in [t_k,t_k+T]\big\} \leq \mathbb{P}\big\{|X(t;0,\xi)| \leq H, \forall t \in [t_k,t_k+T]\big\},
\end{align*}
which implies the desired result by combining with \eqref{eq-4.42}.
\end{proof}

Before we prove the uniform convergence in probability of the norm of the numerical segment process, the following lemma is required.
\begin{lemma}\label{LM511}
    Suppose Assumptions \ref{as1} and \ref{as2} hold. Then for any $\xi \in B(R), \epsilon _8>0, \epsilon _9>0$, there exists a positive integer $j^*$ and $\Delta_3\in(0,\Delta_1\wedge \Delta_2]$ such that for any $\Delta \in (0,\Delta_3]$
    \begin{align*}
        \mathbb{P}\Big \{\sup \limits_{\substack{|s_1-s_2|\leq \tau/j^*\\
s_1,s_2\in [t_k,t_k+\tau]} } |X(s_1)-X(s_2)|\geq \epsilon _9\Big\} \leq \epsilon _8.
    \end{align*}
\end{lemma}
\begin{proof}
    First, it follows from Lemma \ref{LM510} that there exists positive constant $K_3$ large enough such that for any $k\geq -M$
    \begin{align}\label{eq-4.44}
        \mathbb{P}\big\{\overline{\lambda}_k^{K_3} < k+2M\big\} \leq \frac{\epsilon _8}{2}.
    \end{align}
     For any integer $j^*>0$, define $t_{k}^{j^*,j}=t_k+(j\tau)/j^*$, $j=0,1,...,j^*$. Choose $\Delta$ small enough such that $M=j^*\nu$ with some $\nu \in \mathbb{N}^+$. Then, it follows from \eqref{CDSDDE} and the Burkholder–Davis–Gundy inequality that for any $j\in\{0,...,j^*-1\}$
     \begin{align}\label{eq-4.45}
         \notag \mathbb{E}&\Big(\sup \limits_{t\in[t_{k+M}^{j^*,j},t_{k+M}^{j^*,j+1}]}\big(|X(t \wedge t_{\overline{\lambda}_k^{K_3}})-X(t_{k+M}^{j^*,j} \wedge t_{ \overline{\lambda}_k^{K_3}})|^4\big)\Big) \\ \notag
         &\leq  \mathbb{E}\Big(\sup \limits_{t\in[t_{k+M}^{j^*,j},t_{k+M}^{j^*,j+1}]}\big(\big|\sum\limits_{i=0}^{\lfloor(t-t_{k+M+j\nu})/\Delta)\rfloor-1}(X_{(k+M+j\nu+i+1)\wedge\overline{\lambda}_k^{K_3}}-X_{(k+M+j\nu+i)\wedge \overline{\lambda}_k^{K_3}})\\ \notag
         &\quad~+\frac{t-\lfloor t/\Delta\rfloor}{\Delta}(X_{(\lfloor t/\Delta \rfloor+1)\wedge \overline{\lambda}_k^{K_3} }-X_{(\lfloor t/\Delta \rfloor)\wedge \overline{\lambda}_k^{K_3} })\big|^4\big)\Big) \\ \notag\\ \notag
          &\leq 8\Delta^4 \mathbb{E}\Big(\big(\big|\sum\limits_{i=0}^{\nu-1}|f(X_{(k+M+j\nu+i+1)\wedge\overline{\lambda}_k^{K_3}}, X_{(k+M+j\nu+i+1)\wedge \overline{\lambda}_k^{K_3}-M})|\big|^4\big)\Big) \\ \notag
           &\quad~+ 8 \mathbb{E}\Big(\sup \limits_{t\in[t_{k+M}^{j^*,j},t_{k+M}^{j^*,j+1}]}\big(\big|\sum\limits_{i=0}^{\lfloor(t-t_{k+M+j\nu})/\Delta)\rfloor-1}g(X_{(k+M+j\nu+i)\wedge\overline{\lambda}_k^{K_3}},X_{(k+M+j\nu+i)\wedge \overline{\lambda}_k^{K_3}-M})\Delta W_{(k+M+j\nu+i)\wedge\overline{\lambda}_k^{K_3}}\\ \notag
         &\quad~+\frac{t-\lfloor t/\Delta\rfloor}{\Delta}g(X_{(\lfloor t/\Delta \rfloor)\wedge \overline{\lambda}_k^{K_3} }, X_{(\lfloor t/\Delta \rfloor)\wedge \overline{\lambda}_k^{K_3} -M})\Delta W_{(\lfloor t/\Delta \rfloor)\wedge \overline{\lambda}_k^{K_3}}\big|^4\big)\Big) \\ \notag
           &\leq 8\Delta^4 \mathbb{E}\Big(\big(\big|\sum\limits_{i=0}^{\nu-1}|f(X_{(k+M+j\nu+i+1)\wedge\overline{\lambda}_k^{K_3}}, X_{(k+M+j\nu+i+1)\wedge \overline{\lambda}_k^{K_3}-M})|\big|^4\big)\Big) \\ \notag
           &\quad~+ 64\mathbb{E}\Big(\sup \limits_{t\in[t_{k+M}^{j^*,j},t_{k+M}^{j^*,j+1}]}\big(\big|\sum\limits_{i=0}^{\lfloor(t-t_{k+M+j\nu})/\Delta)\rfloor-1}g(X_{(k+M+j\nu+i)\wedge\overline{\lambda}_k^{K_3}},X_{(k+M+j\nu+i)\wedge \overline{\lambda}_k^{K_3}-M})\\ \notag
           &\quad~\times \Delta W_{(k+M+j\nu+i)\wedge\overline{\lambda}_k^{K_3}}\big|^4\big)\Big) \\
         &\quad~+64\mathbb{E}\Big(\sup \limits_{t\in[t_{k+M}^{j^*,j},t_{k+M}^{j^*,j+1}]}\big(\big|\frac{t-\lfloor t/\Delta\rfloor}{\Delta}g(X_{(\lfloor t/\Delta \rfloor)\wedge \overline{\lambda}_k^{K_3} }, X_{(\lfloor t/\Delta \rfloor)\wedge \overline{\lambda}_k^{K_3} -M})\Delta W_{(\lfloor t/\Delta \rfloor)\wedge \overline{\lambda}_k^{K_3}}\big|^4\big)\Big).
     \end{align}
      By Assumption \ref{as1}, there exists a constant $c_{K_3}>0$ such that
    \begin{align}\label{eq-4.46}
        |f(x,y)| \vee |g(x,y)| \leq c_{K_3}
    \end{align}
    for all $|x| \vee |y| \leq K_3$.
    Thus, it follows from \eqref{eq-4.45} and \eqref{eq-4.46} that
    \begin{align}\label{eq-4.47}
          \notag \mathbb{E}&\Big(\sup \limits_{t\in[t_{k+M}^{j^*,j},t_{k+M}^{j^*,j+1}]}\big(|X(t \wedge t_{\overline{\lambda}_k^{K_3}})-X(t_{k+M}^{j^*,j} \wedge t_{ \overline{\lambda}_k^{K_3}})|^4\big)\Big) \\ \notag
         &\leq 8c_{K_3}^4\nu^4 \Delta^4+64c_{K_3}^4\mathbb{E}\Big(\sup \limits_{0\leq j\leq \nu-1}\big(\big|\sum\limits_{i=0}^{j}\Delta W_{(k+M+j\nu+i)\wedge\overline{\lambda}_k^{K_3}}\big|^4\big)\Big) +192c_{K_3}^4\Delta^2 \\ \notag
         &\leq 8c_{K_3}^4\nu^4 \Delta^4+192c_{K_3}^4\Delta^3 + 16384c_{K_3}^4\mathbb{E}\Big(\sum \limits_{i=1}^{\nu-1}|\Delta W_{(k+M+j\nu+i)\wedge\overline{\lambda}_k^{K_3}}|^2\Big)^2\\ \notag
         &\leq 8c_{K_3}^4\nu^4 \Delta^4+192c_{K_3}^4\Delta^3 + 16384c_{K_3}^4 \nu\mathbb{E}\Big(\sum \limits_{i=1}^{\nu-1}|\Delta W_{(k+M+j\nu+i)\wedge\overline{\lambda}_k^{K_3}}|^4\Big)\\ \notag&\leq 8c_{K_3}^4\nu^4 \Delta^4+192c_{K_3}^4\Delta^3 + 16384c_{K_3}^4 \nu^2\Delta^2 \\
          &\leq  \frac{R_2}{(j^*)^2},
    \end{align}
    where $R_2:=8\tau^4 c_{K_3}^4+192c_{K_3}^4\tau^3+16384\tau^2c_{K_3}^4$. Then, for any $\epsilon _9>0$, choose $j^* \geq 1 \vee \big((2\cdot3^4R_2)/(\epsilon _8\epsilon _9^4)\big) $, and by the Chebyshev inequality, we can see that
    \begin{align*}
     \notag   \mathbb{P}&\big\{\overline{\lambda}_k^{K_3} \geq k+2M, \sup \limits_{\substack{|s_1-s_2|\leq \tau/j^*\\
s_1,s_2\in [t_k+\tau,t_k+2\tau]} } |X(s_1)-X(s_2)|\geq \epsilon _9\big\} \\ \notag
&\leq    \mathbb{P}\big\{\overline{\lambda}_k^{K_3} \geq k+2M, 3\max \limits_{0\leq j\leq j^*-1} \sup \limits_{t\in[t_{k+M}^{j^*,j},t_{k+M}^{j^*,j+1}]} |X(t)-X(t_{k+M}^{j^*,j+1})|\geq \epsilon _9\big\} \\ \notag
&\leq \sum \limits_{j=0}^{j^*-1}\mathbb{P}\big\{\overline{\lambda}_k^{K_3} \geq k+2M, \sup \limits_{t\in[t_{k+M}^{j^*,j},t_{k+M}^{j^*,j+1}]} |X(t)-X(t_{k+M}^{j^*,j+1})|\geq \frac{\epsilon _9}{3}\big\} \\ \notag
&\leq\frac{3^4}{\epsilon _9^4}  \sum \limits_{j=0}^{j^*-1}\mathbb{E}\Big(\sup \limits_{t\in[t_{k+M}^{j^*,j},t_{k+M}^{j^*,j+1}]}\big( |X(t)-X(t_{k+M}^{j^*,j+1})|^4\boldsymbol{1}_{\{\overline{\lambda}_k^{K_3} \geq k+2M\}}\big)\Big)\\
&\leq\frac{3^4}{\epsilon _9^4}  \sum \limits_{j=0}^{j^*-1}\mathbb{E}\Big(\sup \limits_{t\in[t_{k+M}^{j^*,j},t_{k+M}^{j^*,j+1}]}\big( |X(t\wedge t_{\overline{\lambda}_k^{K_3}} )-X(t_{k+M}^{j^*,j+1} \wedge t_{\overline{\lambda}_k^{K_3}} )|^4\big)\Big).
    \end{align*}
    This, together with \eqref{eq-4.47} implies that
    \begin{align}\label{eq-4.49}
         \mathbb{P}\big\{\overline{\lambda}_k^{K_3} \geq k+2M, \sup \limits_{\substack{|s_1-s_2|\leq \tau/j^*\\
s_1,s_2\in [t_k+\tau,t_k+2\tau]} } |X(s_1)-X(s_2)|\geq \epsilon _9\big\} \leq \frac{3^4R_2}{\epsilon _9^4j^*} \leq \frac{\epsilon _8}{2}.
    \end{align}
     Further, from \eqref{eq-4.44} and \eqref{eq-4.49}, we have
    \begin{align*}
         \mathbb{P}&\big\{  \sup \limits_{\substack{|s_1-s_2|\leq \tau/j^*\\
s_1,s_2\in [t_k+\tau,t_k+2\tau]} } |X(s_1)-X(s_2)|\geq \epsilon _9\big\} \\ \notag
&\leq \mathbb{P}\big\{\overline{\lambda}_k^{K_3}  < k+2M\big\}+\mathbb{P}\big\{\overline{\lambda}_k^{K_3}  \geq k+2M,  \sup \limits_{\substack{|s_1-s_2|\leq \tau/j^*\\
s_1,s_2\in [t_k+\tau,t_k+2\tau]} } |X(s_1)-X(s_2)|\geq \epsilon _9\big\} \\
&\leq \epsilon _8.
    \end{align*}
\end{proof}

We now prove the convergence of the numerical segment process to the underlying process in probability.
\begin{lemma}\label{LM47}
    Suppose Assumptions \ref{as1} and \ref{as2} hold. Then for any $\epsilon _{10}, \epsilon_{11} >0$ and $k\in \mathbb{N}$, there exists a $\Delta_4 \in (0,\Delta_3]$ such that for any $\Delta \in (0, \Delta_4]$ and $\xi \in B(R)$
    \begin{align*}
        \mathbb{P}\big\{\|X_{t_k}^{0,\xi}-x_{t_k}^{0,\xi}\|> \epsilon _{10}\big\} < \epsilon _{11}
    \end{align*}
\end{lemma}
\begin{proof}
    First, it follows from the Chebyshev inequlity and Theorem \ref{T--2} that for any $\epsilon _{10} >0$
    \begin{align}\label{EQ-568}
 \notag \mathbb{P}&\big\{\big|x(t;0,\xi)-X(t;0,\xi)\big|\geq \frac{\epsilon _{10}}{3}\big\} \\
&\leq \frac{9}{\epsilon _{10}^2}\mathbb{E}\big(\big|x(t;0,\xi)-X(t;0,\xi)\big|^2\big) \leq \frac{9C\Delta}{\epsilon _{10}^2}.
    \end{align}
    Choose a $\nu$ sufficiently large such that
    \begin{align}\label{EQ-569}
        \frac{9C\Delta}{\epsilon _{10}^2} \leq \frac{\nu \Delta \epsilon _{11}}{6\tau}= \frac{\epsilon _{11}}{6j^*}.
    \end{align}
    According to Lemmas \ref{LM511}, \eqref{EQ-568} and \eqref{EQ-569}, we can see that for any $\Delta \in (0,\Delta_4]$
    \begin{align*}
      \notag    \mathbb{P}\big\{\|X_{t_k}^{0,\xi}&-x_{t_k}^{0,\xi}\|> \epsilon _{10}\big\}  \\ \notag
      &=\mathbb{P}\big\{\sup \limits_{0\leq j \leq j^*-1} \sup\limits_{t\in[t_{k-M}^{j^*,j},t_{k-M}^{j^*,j+1}]} \big|X(t;0,\xi)-x(t;0,\xi)\big| \geq \epsilon _{10}\big\}\\ \notag
      &\leq \mathbb{P}\big\{\sup \limits_{0\leq j \leq j^*-1} \sup\limits_{t\in[t_{k-M}^{j^*,j},t_{k-M}^{j^*,j+1}]} \big|X(t;0,\xi)-X(t_{k-M}^{j^*,j};0,\xi)\big| \geq \frac{\epsilon _{10}}{3}\big\}\\ \notag
      &\quad~+\mathbb{P}\big\{\sup \limits_{0\leq j \leq j^*-1}  \big|X(t_{k-M}^{j^*,j};0,\xi)-x(t_{k-M}^{j^*,j};0,\xi)\big| \geq \frac{\epsilon _{10}}{3}\big\} \\ \notag
      &\quad~+\mathbb{P}\big\{\sup \limits_{0\leq j \leq j^*-1} \sup\limits_{t\in[t_{k-M}^{j^*,j},t_{k-M}^{j^*,j+1}]} \big|x(t;0,\xi)-x(t_{k-M}^{j^*,j};0,\xi)\big| \geq \frac{\epsilon _{10}}{3}\big\}\\ \notag
      &\leq \mathbb{P}\big\{  \sup \limits_{\substack{|s_1-s_2|\leq \tau/j^*\\
s_1,s_2\in [t_k-\tau,t_k]} } |X(s_1)-X(s_2)|\geq \frac{\epsilon _{10}}{3}\big\} + j^*\sup \limits_{t \in [t_k-\tau,t_k]}\mathbb{P}\big\{\big|X(t;0,\xi-x(t;0,\xi))\big|\geq \frac{\epsilon _{10}}{3}\big\} \\ \notag
&\quad~+\mathbb{P}\big\{  \sup \limits_{\substack{|s_1-s_2|\leq \tau/j^*\\
s_1,s_2\in [t_k-\tau,t_k]} } |x(s_1)-x(s_2)|\geq \frac{\epsilon _{10}}{3}\big\} <\epsilon _{11}.
    \end{align*}
\end{proof}

We conclude this section by stating the final key result of this work.

\begin{theorem}\label{TM48}
    Under Assumptions \ref{as1}-\ref{as3}, for any $\xi \in B(R)$, the probability measure of $\{X_{t_k}^{0,\xi}\}_{k \geq 0}$ converges to the underlying invariant measure $\pi(\cdot)$ in the Bounded Lipschitz metric $d_{\mathbb{L}}$ as the step size tends to zero, that is
   \begin{align*}
        \lim \limits_{\substack{k\to \infty \\ \Delta \to 0}} d_{\mathbb{L}}(\overline{\mathbb{P}}_{t_k}(\xi, \cdot), \pi(\cdot))=0.
   \end{align*}
\end{theorem}
\begin{proof}
    First, from Lemma \ref{LM3.3}, the process $\{x_t\}_{t \geq 0}$ has a unique invariant measure $\pi$. This means that for any $\epsilon >0$, there exists a $T>0$ such that for any $t_k \geq T$
    \begin{align*}
          d_{\mathbb{L}}(\mathbb{P}_{t_k}(\xi, \cdot), \pi(\cdot)) \leq \frac{\epsilon}{2}.
    \end{align*}
    By Lemma \ref{LM47}, there exists a $\Delta_4 >0$ such that for any $\Delta \in (0, \Delta_4]$
    \begin{align}
       \mathbb{P}\big\{\|X_{t_k}^{0,\xi}-x_{t_k}^{0,\xi}\|> \frac{\epsilon}{4}\big\} \leq \frac{\epsilon}{8}.
    \end{align}
    Next, from the definitions of ~$\mathbb{L}$ and $d_{\mathbb{L}}$, for any $F \in \mathbb{L}$, we have
    \begin{align*}
          d_{\mathbb{L}}(\mathbb{P}_{t_k}(\xi, \cdot), \overline{\mathbb{P}}_{t_k}(\xi, \cdot)) &= \sup \limits_{F \in \mathbb{L}}|\mathbb{E}(F(x_{t_k}^{0,\xi}))-\mathbb{E}(F(X_{t_k}^{0,\xi}))| \leq \mathbb{E}(2 \wedge \|x_{t_k}^{0,\xi}- X_{t_k}^{0, \xi}\|) \\ \notag
          &=\mathbb{E}\Big((2 \wedge \|x_{t_k}^{0,\xi}- X_{t_k}^{0, \xi}\|) \boldsymbol{1}_{\{\|x_{t_k}^{0,\xi}- X_{t_k}^{0, \xi}\| \geq \frac{\epsilon}{4}\}}\Big) \\ \notag
          &\quad~+\mathbb{E}\Big((2 \wedge \|x_{t_k}^{0,\xi}- X_{t_k}^{0, \xi}\|) \boldsymbol{1}_{\{\|x_{t_k}^{0,\xi}- X_{t_k}^{0, \xi}\| <\frac{\epsilon}{4}\}}\Big) \\ \notag
          &\leq 2 \mathbb{P}\big\{\|X_{t_k}^{0,\xi}-x_{t_k}^{0,\xi}\|> \frac{\epsilon}{4}\big\} + \frac{\epsilon}{4} < \frac{\epsilon}{2}.
    \end{align*}
   Finally, by the triangle inequality yields
    \begin{align*}
      d_{\mathbb{L}}(\mathbb{P}_{t_k}(\xi, \cdot), \pi(\cdot))    \leq d_{\mathbb{L}}(\mathbb{P}_{t_k}(\xi, \cdot), \overline{\mathbb{P}}_{t_k}(\xi, \cdot))+d_{\mathbb{L}}(\mathbb{P}_{t_k}(\xi, \cdot), \pi(\cdot)).
    \end{align*}
    The proof is complete.
\end{proof}

\section{Numerical examples}
\begin{figure}[htp!]
    \centering
    \includegraphics[width=0.49\textwidth]{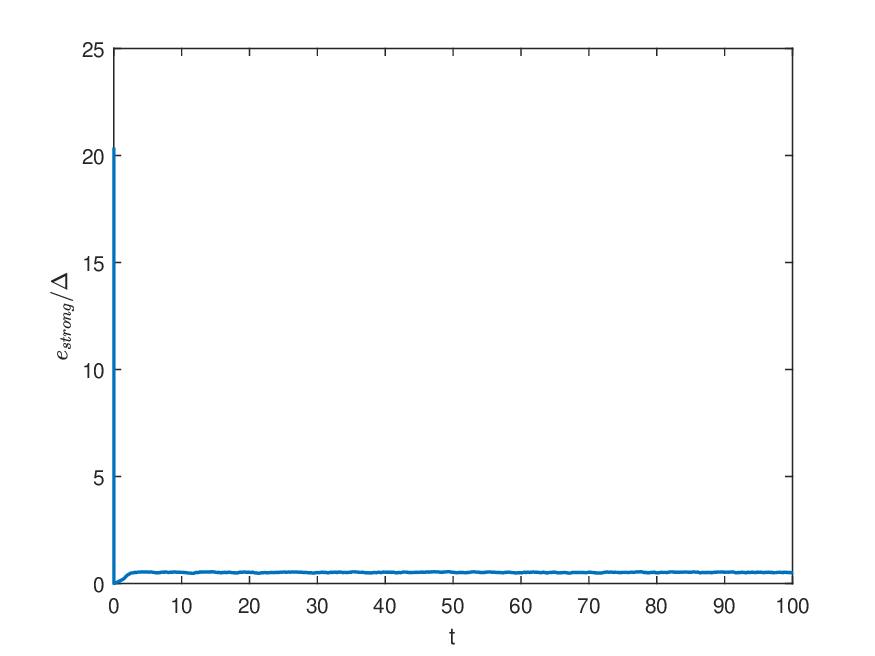}
 \includegraphics[width=0.49\textwidth]{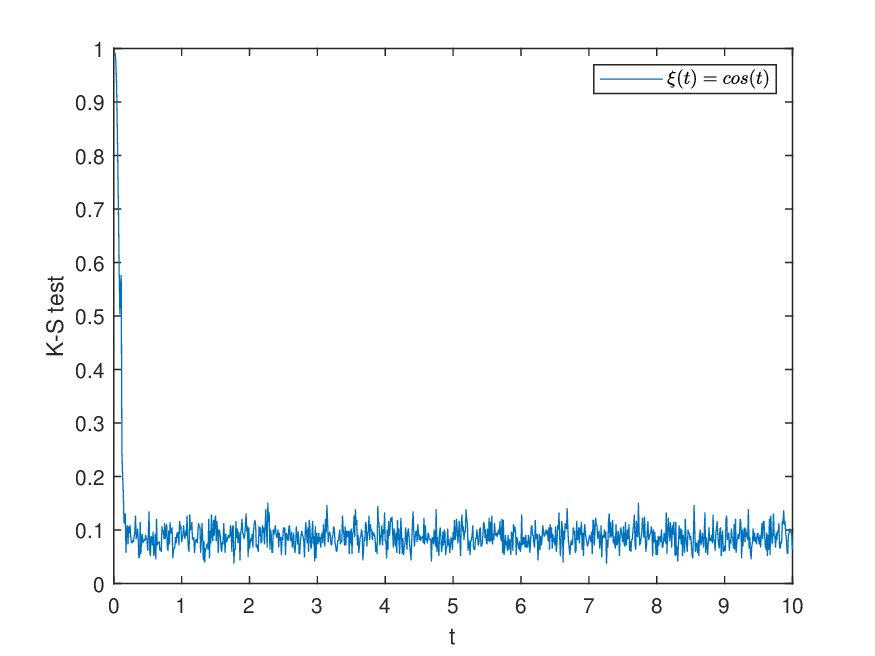}
 \caption{The asymptotic behaviour of the error constant $C$ and the K-S statistics for Example \eqref{EX1};}
    \label{estrong}
\end{figure}

\begin{figure}[htp!]
    \centering
    \includegraphics[width=0.49\textwidth]{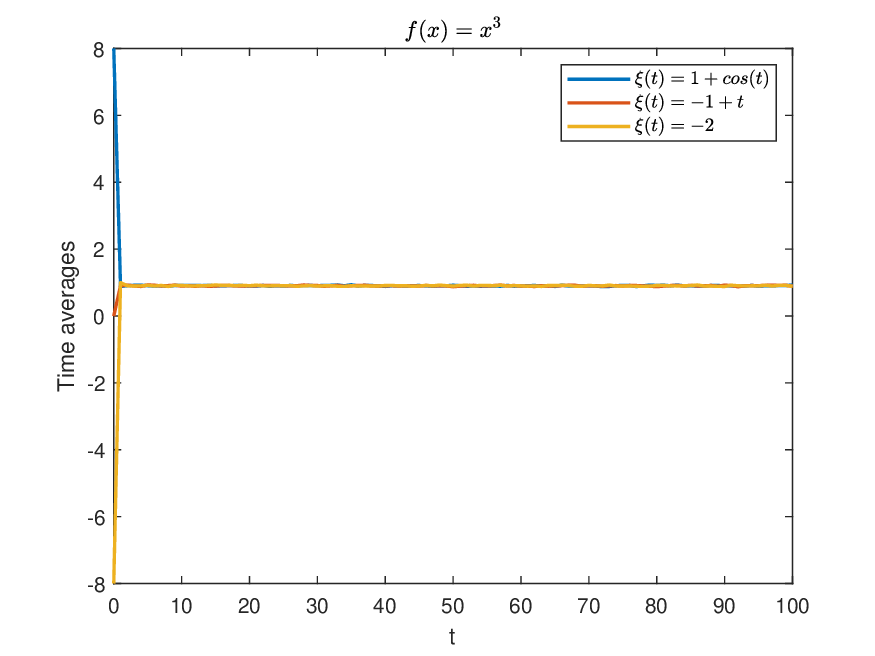}
 \includegraphics[width=0.49\textwidth]{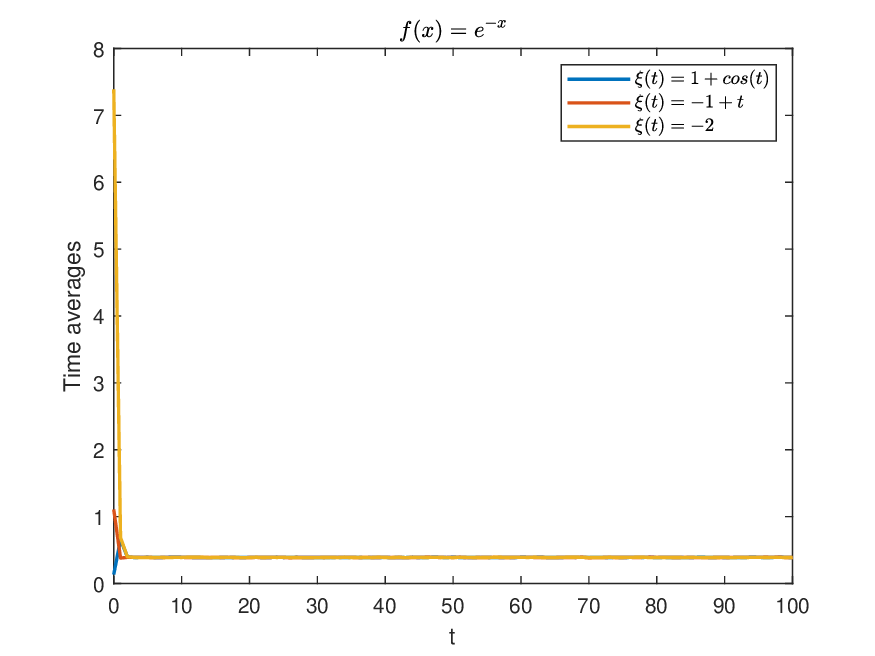}
  \caption{Sample means of $\mathbb{E}(|f(z(k))|)$ $(k\in \mathbb{N})$ with different initial data $\xi$}
    \label{ ergodicity}
\end{figure}
In this section, an example will be presented in order to verify the theories.
\begin{example}
    Consider the following SDDE
   \begin{align}\label{EX1}
    \begin{split}
  \left \{
 \begin{array}{ll}
dx(t)=\big(10-x(t)-10x^3(t)\big)dt + \big(x^2(t-1)\big)dW(t), \quad \quad t >0\\
 x(t)=cos(t), \quad \quad t\in [-1,0].
 \end{array}
 \right.
 \end{split}
 \end{align}
 It is easy to check that Assumptions \ref{as1}, \ref{as2} and \ref{as3} hold with $q=3, a_1=20, \epsilon_1=0,\epsilon_2=1, a_2=a_3=a_4=0,a_5=2,a_6=10,a_7=a_8=0,a_9=1,b_1=2,b_2=0,b_3=10,b_4=18^9/10, b_5=5,b_6=0,b_7=70,b_8=15,b_9=100, b_{10}=1,b_{11}=0,b_{12}=20,b_{13}=2$ and $K_1=1$.

 We approximate the expectation by averaging over 500 paths, using the numerical solution with stepsize $\Delta_1=0.0001$ as the underlying solution $x(t)$ and the stepsize $\Delta_2=0.01$ as the numerical solution $X(t)$. At each time point $t_k=k\Delta_2$, we also approximate $e_{strong}(t_k):=\mathbb{E}(|x(t_k)-X(t_k)|^2)$ across 500 paths. The left figure of Figure \ref{estrong} shows the change of the ratio of error and stepsize $e_{strong}/\Delta$ over the time interval $[0,100]$. It can be seen that the curve does not rise with the increase of time, but fluctuates within a certain range, which is consistent with the theoretical results.

 And by the definition of the segment process, we know that it's a functionally valued random variable. The distribution of the segment process is composed of the
distribution of an infinite and uncountable number of points. We cannot perform numerical experiments on it, and can only study the distribution of finitely countable points on each interval. Therefore, since $\tau=1$, on each interval $[k, k + 1]$, we only discuss the distribution at $t=k$. Similar to before, we take the probability distribution at $t=100$ of the numerical solution with stepsize $\Delta_1=0.0001$ as the invariant measure of the underlying solution. At each time point, the probability distribution is approximated by an empirical distribution of $100$ sample points. As shown in the right part of Figure \ref{estrong}, we use the Kolmogorov-Smirnov test (K-S test) to measure the difference between the empirical distribution and the invariant measure at each time point. The figure shows that the difference tends to be stable, indicating that there is an invariant measure of the underlying solution.

Finally, due to the close relationship between the invariant measure and ergodicity, we also verify the ergodicity of numerical solutions. As shown in Figure \ref{ ergodicity}, we select functions $f(x)=x^3$ and $f(x)=e^{-x}$, and start from three different initial values $1+cos(t)$, $-1+t$, and $-2$. The curves depicting the changes in their time averages $\mathbb{E}(f(z(t)))$ converge together, thus verifying the existence of ergodicity.

\end{example}

\bibliography{References}

\begin{thebibliography}{10}

\bibitem{ACO2023}
L.~Angeli, D.~Crisan, and M.~Ottobre.
\newblock Uniform in time convergence of numerical schemes for stochastic
  differential equations via strong exponential stability: Euler methods,
  split-step and tamed schemes.
\newblock {\em arXiv:2303.15463}, 2023.

\bibitem{AHMG2007}
M.~Arriojas, Y.~Hu, S.-E. Mohammed, and G.~Pap.
\newblock A delayed black and scholes formula.
\newblock {\em Stochastic Anal. Appl.}, 25(2):471 – 492, 2007.

\bibitem{BSY-2023}
J.~Bao, J.~Shao, and C.~Yuan.
\newblock Invariant probability measures for path-dependent random diffusions.
\newblock {\em Nonlinear Anal.}, 228, 2023.

\bibitem{Buckwar2000}
E.~Buckwar.
\newblock Introduction to the numerical analysis of stochastic delay
  differential equations.
\newblock {\em J. Comput. Appl. Math.}, 125(1-2):297 – 307, 2000.

\bibitem{CGMP2019}
F.~Cacace, A.~Germani, C.~Manes, and M.~Papi.
\newblock Predictor-based control of stochastic systems with nonlinear
  diffusions and input delay.
\newblock {\em Automatica}, 107:43 – 51, 2019.

\bibitem{CDO2021}
D.~Crisan, P.~Dobson, and M.~Ottobre.
\newblock Uniform in time estimates for the weak error of the {E}uler method
  for {SDE}s and a pathwise approach to derivative estimates for diffusion
  semigroups.
\newblock {\em Trans. Amer. Math. Soc.}, 374(5):3289--3330, 2021.

\bibitem{DGM2021}
J.~P. Décamps, F.~Gensbittel, and T.~Mariotti.
\newblock Investment timing and technological breakthroughs, 2021.

\bibitem{GA1997}
S.~Grenadier and A.~Weiss.
\newblock Investment in technological innovations: An option pricing approach.
\newblock {\em Journal of Financial Economics}, 44(3):397--416, 1997.

\bibitem{GMY2018}
Q.~Guo, X.~Mao, and R.~Yue.
\newblock The truncated euler–maruyama method for stochastic differential
  delay equations.
\newblock {\em Numer. Algorithms.}, 78(2):599 – 624, 2018.

\bibitem{HMS2003}
D.~J. Higham, X.~Mao, and A.~M. Stuart.
\newblock Exponential mean-square stability of numerical solutions to
  stochastic differential equations.
\newblock {\em LMS J. Comput. Math.}, 6:297 – 313, 2003.

\bibitem{IW1989}
N.~Ikeda and S.~Watanabe.
\newblock {\em Stochastic differential equations and diffusion processes}.
\newblock North-Holland, Amsterdam, 1989.

\bibitem{JY2017}
Y.~Ji and C.~Yuan.
\newblock Tamed em scheme of neutral stochastic differential delay equations.
\newblock {\em J. Comput. Appl. Math.}, 326:337 – 357, 2017.

\bibitem{LM2006}
J.~Lei and M.~C. Mackey.
\newblock Stochastic differential delay equation, moment stability, and
  application to hematopoietic stem cell regulation system.
\newblock {\em SIAM J. Appl. Math.}, 67(2):387 – 407, 2006.

\bibitem{LMS2023}
X.~Li, X.~Mao, and G.~Song.
\newblock Explicit approximation of invariant measure for stochastic delay
  differential equations with the nonlinear diffusion term.
\newblock {\em J. Theoret. Probab.}, 2023.

\bibitem{LCF2004}
M.~Liu, W.~Cao, and Z.~Fan.
\newblock Convergence and stability of the semi-implicit euler method for a
  linear stochastic differential delay equation.
\newblock {\em J. Comput. Appl. Math.}, 171(1-2):255 – 268, 2004.

\bibitem{Mao2007}
X.~Mao.
\newblock Exponential stability of equidistant euler-maruyama approximations of
  stochastic differential delay equations.
\newblock {\em J. Comput. Appl. Math.}, 200(1):297 – 316, 2007.

\bibitem{M2007}
X.~Mao.
\newblock {\em Stochastic Differential Equations and Applications}.
\newblock Horwood, Chichester, UK, 2 edition, 2007.

\bibitem{Mao2015}
X.~Mao.
\newblock Almost sure exponential stability in the numerical simulation of
  stochastic different equations.
\newblock {\em SIAM J. Numer. Anal.}, 53(1):370 – 389, 2015.

\bibitem{MR2005}
X.~Mao and M.~J. Rassias.
\newblock Khasminskii-type theorems for stochastic differential delay
  equations.
\newblock {\em Stochastic Anal. Appl.}, 23(5):1045 – 1069, 2005.

\bibitem{MS2003}
X.~Mao and S.~Sabanis.
\newblock Numerical solutions of stochastic differential delay equations under
  local lipschitz condition.
\newblock {\em J. Comput. Appl. Math.}, 151(1):215 – 227, 2003.

\bibitem{MY2006}
X.~Mao and C.~Yuan.
\newblock {\em Stochastic Differential Equations with Markovian Switching}.
\newblock Imperial College Press, London, 2 edition, 2006.

\bibitem{M1974}
E.~J. McShane.
\newblock {\em Stochastic Calculus and Stochastic Models. Academic}.
\newblock Academic Press, 1974.

\bibitem{M1986}
S.E.A. Mohammed.
\newblock {\em Stochastic Functional Differential Equations}.
\newblock Longman, New York, 1986.

\bibitem{RRV2006}
M.~Reiß, M.~Riedle, and O.~van Gaans.
\newblock Delay differential equations driven by l\'{e}vy processes:
  stationarity and feller properties.
\newblock {\em Stochastic Process. Appl.}, 116(10):1409 --1432, 2006.

\bibitem{SWMF2025}
B.~Shi, Y.~Wang, X.~Mao, and F.~Wu.
\newblock Approximation of invariant measures of a class of backward
  euler-maruyama scheme for stochastic functional differential equations.
\newblock {\em J. Differential Equations}, 389:415 – 456, 2024.

\bibitem{SHGL2022}
G.~Song, J.~Hu, S.~Gao, and X.~Li.
\newblock The strong convergence and stability of explicit approximations for
  nonlinear stochastic delay differential equations.
\newblock {\em Numer. Algorithms}, 89(2):855 – 883, 2022.

\bibitem{WWM2019}
Y.~Wang, F.~Wu, and X.~Mao.
\newblock Stability in distribution of stochastic functional differential
  equations.
\newblock {\em Systems Control Lett.}, 132, 2019.

\bibitem{Zhou2015}
S.~Zhou.
\newblock Strong convergence and stability of backward euler–maruyama scheme
  for highly nonlinear hybrid stochastic differential delay equation.
\newblock {\em Calcolo}, 52(4):445 – 473, 2015.

\end{thebibliography}
  \end{document}